\documentclass[11pt,reqno]{amsart}
\usepackage[T1]{fontenc}
\usepackage[latin9]{inputenc}
\usepackage{multirow}
\usepackage{array}
\usepackage{amsmath, amssymb}
\usepackage{amsthm} 

 \usepackage{bm}

\theoremstyle{plain} 
\newtheorem{theorem}{Theorem}[section]
\newtheorem{lemma}[theorem]{Lemma}
\newtheorem{corollary}[theorem]{Corollary}
\newtheorem{proposition}[theorem]{Proposition}
\newtheorem{fact}[theorem]{Fact}
\newtheorem{theor}{Theorem}

\theoremstyle{definition}
\newtheorem{definition}[theorem]{Definition}
\newtheorem{remark}[theorem]{Remark}
\newtheorem{example}[theorem]{Example}

 \usepackage{caption}

\usepackage{color}
\usepackage{graphicx,color}
\usepackage{amsmath, amssymb, graphics}

\graphicspath{{./Figures/}}

\DeclareMathOperator{\tr}{tr}
\begin{document}
\title[Behavior of the Gaussian curvature of timelike minimal surfaces]{Behavior of the Gaussian curvature of timelike minimal surfaces with singularities}
  \author[S. Akamine]{Shintaro Akamine}
   \address{Graduate School of Mathematics Kyushu University\\
744 Motooka, Nishi-ku,\
Fukuoka 819-0395, Japan}
 \email{s-akamine@math.kyushu-u.ac.jp}


\begin{abstract}We prove that the sign of the Gaussian curvature of any timelike minimal surface in the $3$-dimensional Lorentz-Minkowski space is determined by the degeneracy and  the orientations of the two null curves that generate the surface. We also investigate the behavior of the Gaussian curvature near singular points of a timelike minimal surface which admits some kind of singular points.
\end{abstract}

\maketitle

\section{Introduction}\label{Sec.1}A timelike surface in the 3-dimensional Lorentz-Minkowski space $\mathbb{L}^3$ is a surface whose first fundamental form is a Lorentzian metric. In contrast with surfaces in the 3-dimensional Euclidean space $\mathbb{E}^3$ and spacelike surfaces in $\mathbb{L}^3$, timelike surfaces do not always have real principal curvatures, that is, their shape operators are not always diagonalizable even over the complex number field $\mathbb{C}$. In general the diagonalizability of the shape operator of a timelike surface is determined by the discriminant of the characteristic equation for the shape operator, which is $H^2-K$ where $H$ is the mean curvature and $K$ is the Gaussian curvature of the considered timelike surface. In this paper we study the case that $H$ vanishes identically. 

A timelike surface whose mean curvature vanishes identically is called a {\it timelike minimal surface}, and McNertney \cite{McNertney} proved that any such surface can be expressed as the sum of two {\it null curves} (see also Fact \ref{Fact:McNertney}), where a null curve is a regular curve whose velocity vector field is lightlike. 
Based on the studies \cite{FSUY, UY} for spacelike case, Takahashi \cite{T} introduced a notion of timelike minimal surfaces with some kind of singular points of rank one, which are called {\it minfaces} (see Definition \ref{def:minface0} and Definition \ref{def:minface}). 
He also gave criteria for cuspidal edges, swallowtails and cuspidal cross caps which appear frequently on minfaces.
 
 The diagonalizability of the shape operator of a timelike minimal surface is determined by the sign of the Gaussian curvature $K$. More precisely the shape operator is diagonalizable over  the real number field $\mathbb{R}$ on points with negative Gaussian curvature and diagonalizable over $\mathbb{C}\setminus \mathbb{R}$ on points with positive Gaussian curvature. Flat points consist of umbilic points and {\it quasi-umbilic points} (see Definition \ref{def:quasi-umbilic}). Therefore the problem of the diagonalizability of the shape operator of a timelike minimal surface is reduced to the problem of the sign of the Gaussian curvature. This would be quite different from minimal surfaces in $\mathbb{E}^3$, which have non-positive Gaussian curvature, and from maximal surfaces in $\mathbb{L}^3$, which have non-negative Gaussian curvature. Hence, to determine the sign of the Gaussian curvature of timelike minimal surfaces is an important problem. In this paper we investigate how to determine the sign of the Gaussian curvature of a timelike minimal surface near regular and singular points.
 
To achieve our goal, we first give a characterization of flat points of a timelike minimal surface by the notion of non-degeneracy of null curves which generate the surface (Proposition \ref{prop:flat points}). Near non-flat points, we also prove that the sign of the Gaussian curvature is determined only by the {\it orientations} of two generating null curves (see Definition \ref{ori}).  
In addition to this result, by using the notion of pseudo-arclength parameters of null curves, we can also give a construction method of conformal curvature line coordinate systems and conformal asymptotic coordinate systems near non-flat points according to the sign of the Gaussian curvature of a timelike minimal surface (Theorem \ref{theorem:p-Weierstrass_real}).

About the behavior of the Gaussian curvature near singular points of surfaces in an arbitrary 3-dimensional Riemannian manifold, some notions of curvatures along singular points of {\it frontals} and {\it wave fronts} (or {\it fronts} for short, and the definitions of frontals and fronts are given in Section \ref{Sec.4}) were introduced in \cite{MSUY, SUY09}, and many relations between the behaviors of these curvatures and the Gaussian curvature along singular points of frontals and fronts were revealed in \cite{FSUY, MSUY, SUY09}. On the other hand, in $\mathbb{L}^3$, Takahashi \cite{T} proved that any minface is a frontal and gave a necessary and sufficient condition for a minface to be a front (see Fact \ref{Lemma T-2}). Based on these backgrounds, we prove the following result:
\begin{theor}\label{Maintheorem}
Let $f: \Sigma \longrightarrow \mathbb{L}^3$ be a minface and $p\in \Sigma$ a singular point of $f$. 
\begin{itemize}
\item[(i)] If $p$ is a cuspidal edge, then there is no umbilic point near $p$.
\item[(ii)] If $f$ is a front at $p$ and $p$ is not a cuspidal edge, then there is no umbilic and quasi-umbilic points near $p$. Moreover the Gaussian curvature $K$ is negative near $p$ and $\displaystyle \lim_{q \to p} K(q)=-\infty$.
\item[(iii)] If $f$ is not a front at $p$ and $p$ is a non-degenerate singular point, then there is no umbilic and quasi-umbilic points near $p$. Moreover the Gaussian curvature $K$ is positive near $p$ and $\displaystyle \lim_{q \to p} K(q)=\infty$.
\end{itemize}
\end{theor}

On the Gaussian curvature near cuspidal edges, Saji, Umehara and Yamada pointed out in \cite{SUY09} that the shape of singular points is restricted when the Gaussian curvature is bounded. In \cite{SUY09}, they introduced the {\it singular curvature} on cuspidal edges, and proved that if the Gaussian curvature with respect to the induced metric from $\mathbb{E}^3$ is bounded and positive (resp.\ non-negative) near a cuspidal edge, then the singular curvature is negative (resp.\ non-positive). Noting that for a timelike surface the Gaussian curvatures with respect to the induced metrics from $\mathbb{E}^3$ and $\mathbb{L}^3$ have opposite signs, we can prove the following statement for minfaces:
\begin{theor}\label{Maintheorem2}
The Gaussian curvature with respect to the induced metric from $\mathbb{L}^3$ near a cuspidal edge on a minface and the singular curvature have the same sign. 
\end{theor}
In fact we will prove a stronger result (Theorem \ref{Theorem: Gaussian curvature near cusps}) than Theorem \ref{Maintheorem2}.
By Theorems \ref{Maintheorem} and \ref{Maintheorem2}, we obtain criteria for the sign of the Gaussian curvature near any non-degenerate singular point on a minface.

This article is organized as follows: In Section \ref{Sec.2} we describe some notions of timelike surfaces and null curves. We also give the definition of minfaces as a class of timelike minimal surfaces with singular points by using a representation formula derived in \cite{T}. In Section \ref{Sec.3} we investigate the behavior of the Gaussian curvature near regular points. Finally, in Section \ref{Sec.4} we discuss the sign of the Gaussian curvature near singular points on minfaces and prove our main results: Theorem \ref{Maintheorem} and Theorem \ref{Theorem: Gaussian curvature near cusps}. In Appendix \ref{Appendix:A} we review a precise description of geometry of minfaces given in Takahashi's Master thesis \cite{T}.

\textbf{Acknowledgement.} The author is grateful to Professor Atsufumi Honda for his valuable comments and fruitful discussion. He is also grateful to Professor Miyuki Koiso for her encouragement and suggestions. This work was supported by Grant-in-Aid for JSPS Fellows Number 15J06677.

\section{Preliminaries\label{Sec.2}}
\subsection{Timelike surfaces and their shape operators\label{preliminary1}}
 We denote by $\mathbb{L}^3$ the 3-dimensional Lorentz-Minkowski space, that is, the 3-dimensional real vector space $\mathbb{R}^3$ with the Lorentzian metric \begin{center}
$\langle \ ,\ \rangle =-(dx^0)^2+(dx^1)^2+(dx^2)^2$,
\end{center}
where $(x^0, x^1, x^2)$ are the canonical coordinates in $\mathbb{R}^3$. In $\mathbb{L}^3$, a vector $v$ has one of the three {\it causal characters}: it is {\it spacelike} if $\langle v, v \rangle > 0$ or $v=0$, {\it timelike} if  $\langle v, v \rangle < 0$, and {\it lightlike} or {\it null} if $\langle v, v \rangle = 0$ and $v\neq 0$.  We denote the set of null vectors by $\mathbb{Q}^2:= \left\{v=(v^0, v^1, v^2)\in \mathbb{L}^3\mid \langle v,v \rangle=0,\, v^0\neq0\right\}$ and call it the {\it lightcone}. Let $\Sigma :=\Sigma ^2$ be a two-dimensional connected and oriented smooth manifold and $f$ : $\Sigma $ $\longrightarrow$ $\mathbb{L}^3$ be an immersion.   An immersion $f$ is said to be {\it timelike} (resp. {\it spacelike}) if the first fundamental form, that is, the induced metric $\mathrm{I}=f^*\langle \ ,\  \rangle$ is Lorentzian (resp. Riemannian) on $\Sigma$.

 For a timelike immersion $f$ and its spacelike unit normal vector field $\nu$, the shape operator $S$ and the second fundamental form ${\rm II}$ are defined as\begin{align}
 df(S(X))=-\overline{\nabla}_{X} \nu,\quad & {\rm II}(X,Y)=\langle \overline{\nabla}_{df(X)}df(Y)-df(\nabla_{X}Y), \nu \rangle,\nonumber 
\end{align}
where $X$ and $Y$ are smooth vector fields on $\Sigma $, and $\nabla$, $\overline{\nabla}$ are the Levi-Civita connections on $\Sigma$ and $\mathbb{L}^3$, respectively.  The mean curvature $H$ and the Gaussian curvature $K$ are defined as $H=(1/2)\tr{{\rm II}}$ and $K=\det S$. Let $\tilde{K}$ be the sectional curvature of the Lorentzian manifold $(\Sigma, \rm{I})$. Then the Gauss equation 
\begin{equation*}
\tilde{K}=K
\end{equation*}
implies that the Gaussian curvature $K$ is intrinsic.

One of the most important differences between spacelike surfaces and timelike surfaces is the diagonalizability of the shape operator, that is, the shape operator of a timelike surface is not always diagonalizable even over $\mathbb{C}$. For surfaces in $\mathbb{E}^3$ and spacelike surfaces in $\mathbb{L}^3$, the Gaussian curvature $K$ and mean curvature $H$ satisfy $H^2 -K \geq 0$, and the equality holds on umbilic points, where an {\it umbilic point} of a surface is a point on which the second fundamental form ${\rm II}$ is a scalar multiple of the first fundamental form ${\rm I}$. On the other hand, there are three possibilities of the diagonalizability of the shape operator of a timelike surface in $\mathbb{L}^3$ as follows:
 \begin{enumerate}
\item[(i)] The shape operator is diagonalizable over $\mathbb{R}$. In this case $H^2 - K \geq 0$ with the equality holds on umbilic points.
\item[(ii)] The shape operator is diagonalizable over $\mathbb{C}\setminus \mathbb{R}$. In this case $H^2 - K < 0$. 
\item[(iii)] The shape operator is non-diagonalizable over $\mathbb{C}$. In this case $H^2 - K = 0$.
\end{enumerate}
About Case (iii), Clelland \cite{Clelland} introduced the following notion:
\begin{definition}[\cite{Clelland}]\label{def:quasi-umbilic}
A point $p$ on a timelike surface $\Sigma$ is called {\it quasi-umbilic} if the shape operator of $\Sigma$ is non-diagonalizable over $\mathbb{C}$. 
\end{definition}

\subsection{Timelike minimal surfaces and minfaces\label{preliminary2}}
 For a timelike surface $f$ : $\Sigma $ $\longrightarrow$ $\mathbb{L}^3$, near each point, we can take a {\it Lorentz isothermal coordinate system} $(x,y)$, that is, the first fundamental form $\mathrm{I}$ is written as $\mathrm{I}=E(-dx^2+dy^2)$ with a non-zero function $E$, and a {\it null coordinate system} $(u,v)$ that is, $\mathrm{I}$ is written as $\mathrm{I}=2\Lambda dudv$. A curve $\gamma$ in $\mathbb{L}^3$ 
 whose velocity vector field $\gamma'$ is lightlike is called a {\it null curve}, and a null coordinate system is a coordinate system on which the image of coordinate curves are null curves. Moreover, up to constant multiple, there is a one-to-one correspondence between these coordinate systems as follows:
 \begin{equation*}
 x=\frac{u-v}{2},\quad y=\frac{u+v}{2}.
\end{equation*}

On each null coordinate system $(u,v)$, an immersion $f$ and its mean curvature $H$ satisfy $H\nu=\frac{2}{\Lambda}\frac{\partial^2 f}{\partial u \partial v}$. Therefore, we obtain the following well-known representation formula.
\begin{fact}[\cite{McNertney}]\label{Fact:McNertney}
If $\varphi(u)$ and $\psi(v)$ are null curves in $\mathbb{L}^3$ such that $\varphi'(u)$ and $\psi'(v)$ are linearly independent for all $u$ and $v$, then
\begin{equation}\label{null curves decomposition}
f(u,v)=\frac{\varphi(u)+\psi(v)}{2}
 \end{equation}
gives a timelike minimal surface. Conversely, any timelike minimal surface can be written locally as the equation (\ref{null curves decomposition}) with two null curves $\varphi$ and $\psi$.
\end{fact}

In this paper, we consider the following class of timelike minimal surfaces with singular points of rank one, which was introduced in \cite{T} (see also Definition \ref{def:minface} in Appendix \ref{Appendix:A}):
\begin{definition}\label{def:minface0}A smooth map $f:\Sigma \longrightarrow \mathbb{L}^3$ is called a {\it minface} if at each point of $\Sigma$ there exists a local coordinate system $(u,v)$ in a domain $U$, functions $g_1=g_1(u)$, $g_2=g_2(v)$, and 1-forms $\omega_1=\hat{\omega}_1(u)du$, $\omega_2=\hat{\omega}_2(v)dv$ with $g_1(u)g_2(v)\neq1$ on an open dense set of $U$ and $\hat{\omega}_1\neq 0$, $\hat{\omega}_2\neq 0$ at each point on $U$ such that $f$ can be decomposed into two null curves:
\begin{align}\label{eq:p-Weierstrass_real1}
  f(u,v)
  &=\frac{1}{2}\int^u_{u_0}\left( -1-(g_1)^2,1-(g_1)^2, 2g_1 \right)\omega_1\nonumber\\
  &+\frac{1}{2}\int^v_{v_0}\left( 1+(g_2)^2,1-(g_2)^2, -2g_2 \right)\omega_2 
    +f(u_0,v_0).
\end{align}
We denote these two null curves by $\varphi=\varphi(u)$ and $\psi=\psi(v)$. The quadruple $(g_1, g_2,\omega_1, \omega_2)$ is called the {\it real Weierstrass data}.
\end{definition}
A {\it singular point} of a minface $f$ is a point of $\Sigma$ on which $f$ is not immersed, and the set of singular points on $U$ of a minface $f$ corresponds to the set $\{(u,v)\in U\mid g_1(u)g_2(v)=1\}$. 
\begin{remark}
In \cite{T}, Takahashi originally gave the notion of minfaces as Definition \ref{def:minface} in Appendix \ref{Appendix:A} by using the notion of para-Riemann surfaces. To study the local behavior of the Gaussian curvature near singular points of timelike minimal surfaces, we adopt the above definition. In Appendix \ref{Appendix:A}, we prove the representation formula (\ref{eq:p-Weierstrass_real1}) from the original definition of minfaces (Fact \ref{theorem:p-Weierstrass}) and give a precise description of geometry of minfaces.
\end{remark}
\subsection{Null curves}\label{subsection:null_curves}
In this subsection, we describe some notions of null curves. 
\begin{definition}[cf. \cite{FKKRSUYY_Okayama,O}]

A null curve $\gamma=\gamma(t)$ in $\mathbb{L}^3$ 
 is called {\it degenerate} or {\it non-degenerate at $t$} if $\gamma'\times \gamma''= 0$ or $\gamma'\times \gamma''\neq 0$ at $t$, respectively. If $\gamma$ is non-degenerate everywhere, it is called a {\it non-degenerate} null curve.
\end{definition} 
A null curve which is degenerate everywhere is a straight line with a lightlike direction. As pointed out in Section 2 in \cite{O}, the non-degeneracy of a null curve is characterized by the following conditions.

\begin{lemma}[cf. \cite{O}]\label{lemma:non-deg}
For a null curve $\gamma=\gamma(t)$ in $\mathbb{L}^3$ 
the following (i), (ii) and (iii) are equivalent:
\begin{itemize}
\item[(i)] $\gamma$ is non-degenerate at $t$,
\item[(ii)] $\gamma''(t)$ is a non-zero spacelike vector, that is, $\langle \gamma''(t), \gamma''(t) \rangle >0$,
\item[(iii)] $\det\left(\gamma'(t)\ \gamma''(t)\ \gamma'''(t)\right)\neq 0$.
\end{itemize}\end{lemma}

By Lemma \ref{lemma:non-deg}, we can introduce the following notions for non-degenerate null curves.

\begin{definition}[\cite{B, V}]
For a non-degenerate null curve $\gamma=\gamma(t)$, a parameter $t$ is called a {\it pseudo-arclength parameter} of $\gamma$ if  $\langle \gamma''(t), \gamma''(t) \rangle \equiv 1$.
\end{definition}

\begin{definition}\label{ori}
We define the {\it orientation} of a non-degenerate null curve $\gamma$ by the sign of $\det\left(\gamma'\ \gamma''\ \gamma'''\right)$.
\end{definition}
\begin{remark}\label{remark:ori}
If we take a pseudo-arclength parameter $s$, then $\det\left(\gamma'\ \gamma''\ \gamma'''\right)=\pm1$, which represents the orientation of $\gamma$. 
Moreover, the orientation of a non-degenerate null curve has the following geometric meaning: If we consider the projection of $\gamma'$, which is on the lightcone $\mathbb{Q}^2$, into the time slice $x^0=1$, then the projected curve on $\mathbb{S}^1= \{(1,x^1,x^2) \mid (1,x^1,x^2)\in \mathbb{Q}^2\}$ is anticlockwise if the orientation is positive, and clockwise if the orientation is negative as $x^0$ increases. See Figure \ref{Fig1} and Remark \ref{Milnor's study}. 

\begin{figure}[!h]
\begin{center}
\begin{tabular}{c}
\hspace{+2.0cm}
\begin{minipage}{0.4\hsize}
\begin{center}
\vspace{-0.8cm}
\includegraphics[clip,scale=0.30,bb=0 0 555 449]{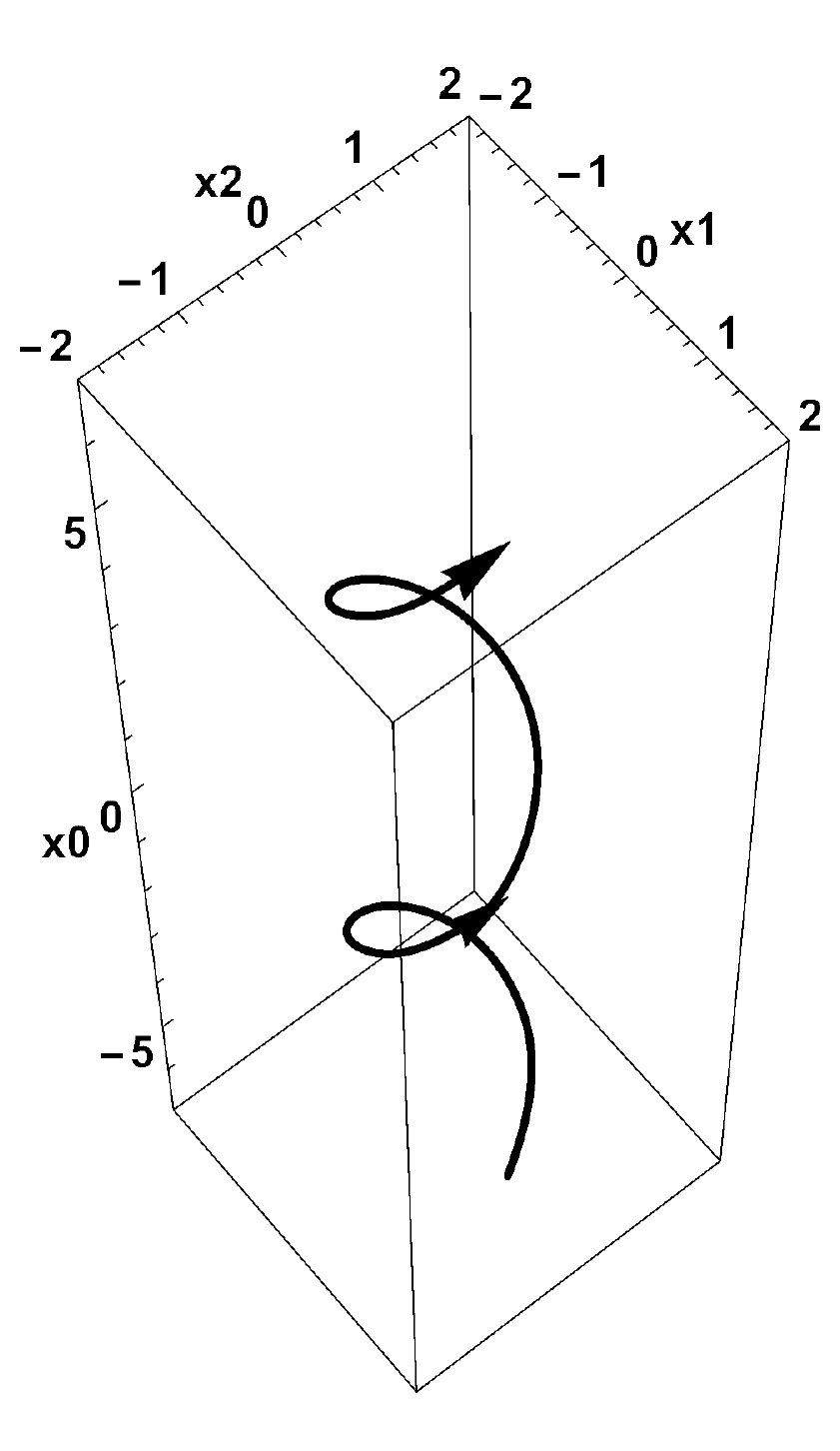}
\vspace{0.5cm}
\end{center}
\end{minipage}
\hspace{0.3cm}
\begin{minipage}{0.4\hsize}
\begin{center}
\vspace{-0.45cm}
\includegraphics[clip,scale=0.30,bb=0 0 555 449]{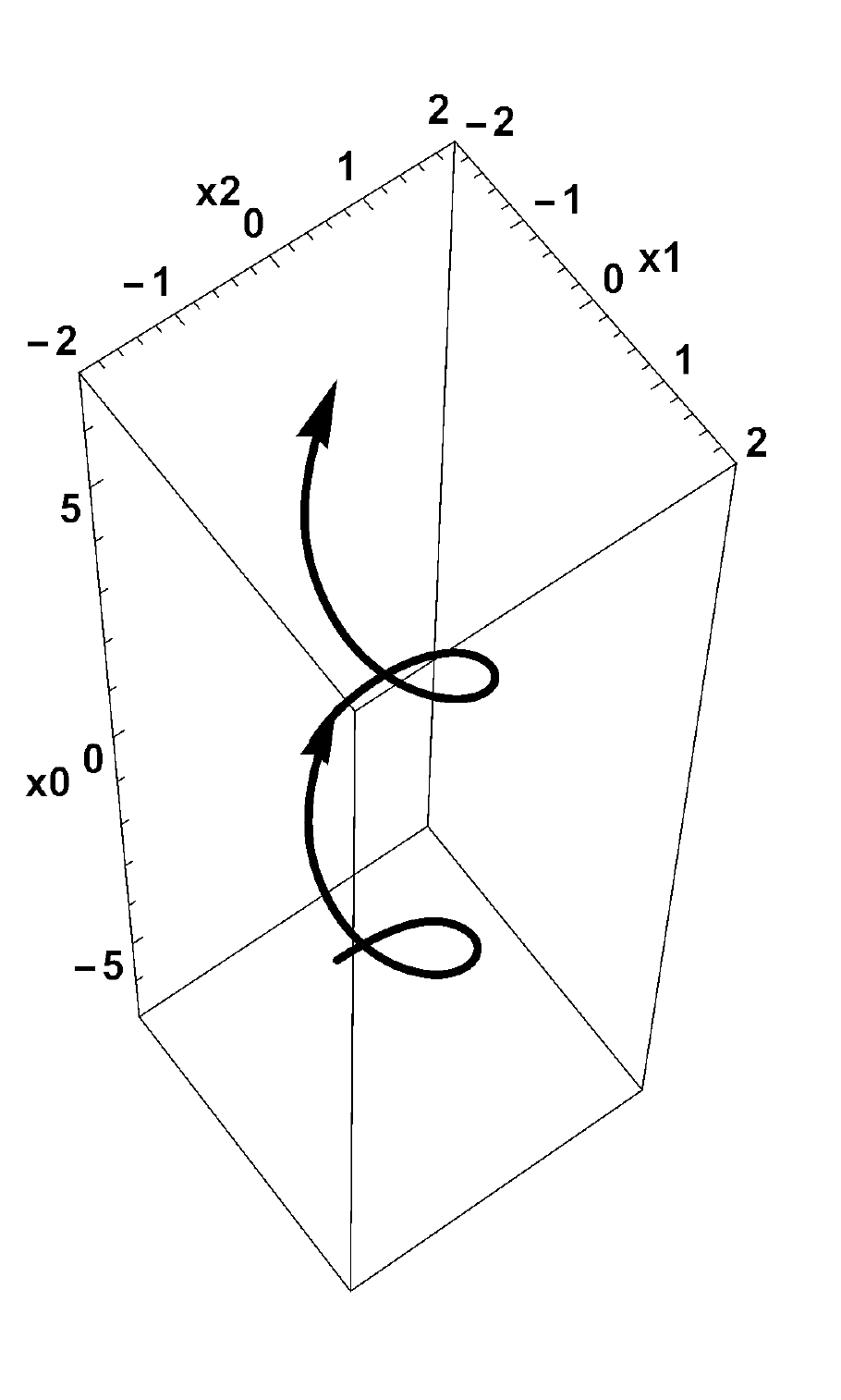}
\vspace{0.3cm}
\end{center}
\end{minipage}

\end{tabular}
\end{center}
\vspace{-0.8cm}
\caption{Examples of non-degenerate null curves with positive (the left figure) and negative orientation (the right figure).}\label{Fig1}

\end{figure}
\end{remark}

\section{The sign of the Gaussian curvature and orientations of null curves}\label{Sec.3}
In this section we give a characterization of flat points and investigate the sign of the Gaussian curvature of minfaces by using the notions of degeneracy and orientations of null curves. 
 \subsection{A characterization of flat points}

As we saw in Section \ref{Sec.2}, flat points on each minface consist of umbilic and quasi-umbilic points. First, we give a characterization of flat points of a minface from a viewpoint of null curves.

\begin{proposition}\label{prop:flat points}
Let $p$ be a regular point in a minface $f$. Then the following statements hold:
\begin{enumerate}
\item[(i)] $p$ is an umbilic point of $f$ if and only if the two null curves in the equation (\ref{null curves decomposition}) are degenerate at $p$.
\item[(ii)] $p$ is a quasi-umbilic point of $f$ if and only if only one of the two null curves in the equation (\ref{null curves decomposition}) is degenerate at $p$.
\end{enumerate}
\end{proposition}

\begin{proof}
If we take a null coordinate system $(u,v)$ on which $f$ is written as (\ref{null curves decomposition}), then the first and the second fundamental forms can be written as follows:
\begin{center}
$\mathrm{I}=2\Lambda dudv$\quad and\quad $\mathrm{II}=Qdu^2+Rdv^2$.
\end{center}
Therefore, the shape operator is 
\begin{equation}\label{condition1}
\hspace{0.1cm} S=\mathrm{I}^{-1}\mathrm{II}=
\left(\begin{array}{cc} 0 & \frac{1}{\Lambda} \\ \frac{1}{\Lambda} & 0 \\ \end{array} \right)
\left(\begin{array}{cc} Q & 0 \\ 0 & R \\ \end{array} \right)=
\left(\begin{array}{cc} 0 & \frac{R}{\Lambda} \\ \frac{Q}{\Lambda} & 0 \\ \end{array} \right).
\end{equation}

On the other hand, we can see that there exists a real number $a$ such that
\begin{center}
$2f_{uu}(p)=\varphi''(p)=a\varphi'(p)+2Q(p)\nu(p)$.
\end{center}
Therefore, $\varphi$ is degenerate at $p$ if and only if $ Q(p)=0$. By using (\ref{condition1}), we obtain the desired result.
 \end{proof}
 \begin{remark}
The differential coefficients $Q$ and $R$ are called (coefficients of) {\it Hopf differentials} on a timelike surface, which was introduced in \cite{IT}. The degenerate points of two null curves $\varphi$ and $\psi$ correspond to zeros of these Hopf differentials $Q$ and $R$, respectively.
 \end{remark}
\begin{example}\label{ex:K-change}
Let us take the two null curves $\varphi$ and $\psi$
\begin{equation*}
\varphi(u)=\left(u+\frac{u^5}{5}, \frac{2}{3}u^3, u-\frac{u^5}{5}\right),\quad \psi(v)=\left(-v-\frac{v^5}{5}, \frac{2}{3}v^3, v-\frac{v^5}{5}\right),
\end{equation*}
which are degenerate at the origin, and consider the timelike minimal surface constructed by the equation (\ref{null curves decomposition}). The Gaussian curvature $K$ of this surface is $K=-\frac{4uv}{(1+u^2v^2)^8}$. Proposition \ref{prop:flat points} states that the set of flat points of this surface consists of quasi-umbilic points except the intersection and the intersection is an umbilic point. See Figure \ref{Fig2}. As this example, the quasi-umbilic points (and also the umbilic points) of a timelike minimal surface are not isolated in general. 
\begin{figure}[htbp]
 \vspace{-1.4cm}
\begin{center}
\includegraphics[clip,scale=0.30,bb=0 0 610 350]{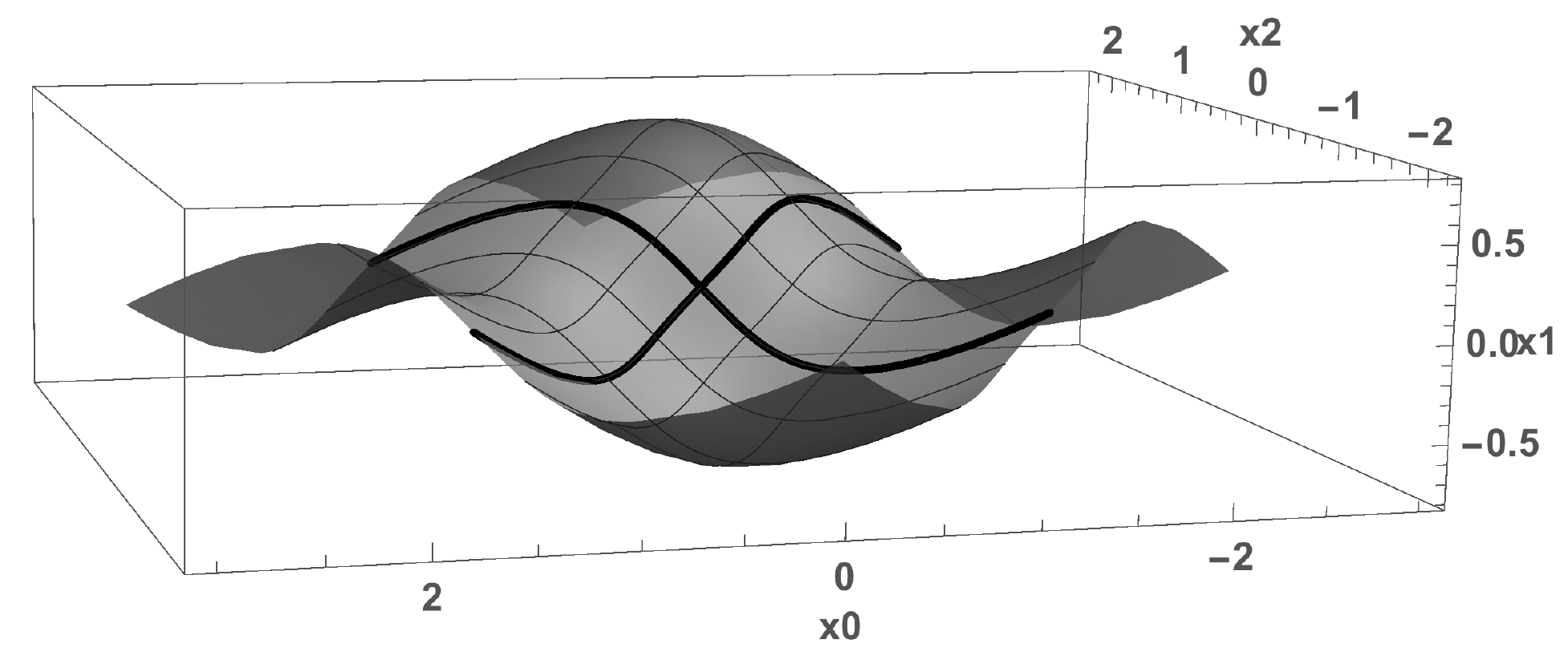}
\caption{An example on which the sign of the Gaussian curvature changes along quasi-umbilic curves (black curves except the intersection).}\label{Fig2}
\end{center}
\end{figure}
\end{example}

 \subsection{The sign of the Gaussian curvature near non-flat points}
In the previous subsection we gave a characterization of flat points using the notion of degeneracy of null curves of a minface. In this subsection we investigate how to determine the sign of the Gaussian curvature and give a construction method of conformal curvature line (resp. conformal asymptotic) coordinate systems near non-flat points of a minface based on the study by Takahashi \cite{T}.

First we consider the two null curves $\varphi=\varphi(u)$ and $\psi=\psi(v)$ in the equation (\ref{eq:p-Weierstrass_real1}). Away from flat points, the two null curves $\varphi$ and $\psi$ are non-degenerate by Proposition \ref{prop:flat points}, and hence we can take pseudo-arclength parameters of $\varphi$ and $\psi$ near non-flat points. Since $\langle \varphi'', \varphi'' \rangle =4{g_1'}^2\hat{\omega}_1^2$, $\hat{\omega}_1\neq 0$, and (ii) of Lemma \ref{lemma:non-deg}, $g_1'\neq0$ near each non-flat point. Moreover, the parameter $u$ is a pseudo-arclength parameter of $\varphi$ if and only if $\hat{\omega}_1$ and $g_1$ satisfy 
\begin{equation*}\label{eq:pseudo-arc-phi}
g_1'\hat{\omega}_1(u)=-\frac{\varepsilon_\varphi}{2},\quad \varepsilon_\varphi=\pm1.
\end{equation*}

After a straightforward calculation, we obtain the equation $\det(\varphi', \varphi'', \varphi''')=\varepsilon_\varphi$, that is, $\varepsilon_\varphi$ is nothing but the orientation of $\varphi$ which was introduced in Definition \ref{ori}. Similarly, the parameter $v$ is a pseudo-arclength parameter of $\psi$ if and only if $\hat{\omega}_2$ and $g_2$ satisfy 
\begin{equation*}\label{eq:pseudo-arc-psi}
g_2'\hat{\omega}_2(v)=-\frac{\varepsilon_\psi}{2},\quad \varepsilon_\psi=\pm1,
\end{equation*}
and $\varepsilon_\psi$ also represents the orientation of $\psi$. Therefore we obtain the following formula near non-flat points
\begin{align}\label{eq:p-Weierstrass_real}
  f(u,v)
  &=\frac{1}{2}\int^u_{u_0}\left( -1-(g_1)^2, 1-(g_1)^2, 2g_1 \right)\frac{-\varepsilon_\varphi}{2g'_1}du \nonumber\\
  &+\frac{1}{2}\int^v_{v_0}\left( 1+(g_2)^2, 1-(g_2)^2, -2g_2 \right)\frac{-\varepsilon_\psi}{2g'_2}dv  
    +f(u_0,v_0).
\end{align}
From now on, we consider the Lorentz isothermal coordinate system $(x, y)=(\frac{u-v}{2}, \frac{u+v}{2})$ associated to the null coordinate system $(u, v)$ constructed from pseudo-arclength parameters of $\varphi$ and $\psi$. On the coordinate system, the first and the second fundamental forms $\mathrm{I}$ and $\mathrm{II}$ can be written as follows:
\begin{equation*}\label{eq:I and II}
\mathrm{I} =  \frac{\varepsilon_\varphi \varepsilon_\psi}{4g_1'g_2'}(1-g_1g_2)^2(-dx^2+dy^2), \quad \mathrm{II} =(\frac{\varepsilon_\varphi}{2}-\frac{\varepsilon_\psi}{2})(dx^2+dy^2)+(\varepsilon_\varphi+\varepsilon_\psi)dxdy.
\end{equation*}
We denote the conformal factor $\frac{\varepsilon_\varphi \varepsilon_\psi}{4g_1'g_2'}(1-g_1g_2)^2$ by $E$. Then the Gaussian curvature $K$ of the minface is written as
\begin{equation}\label{eq:Gauss}
K=\frac{\varepsilon_\varphi \varepsilon_\phi}{E^2}.
\end{equation}
Therefore, the sign of the Gaussian curvature of the non-flat points of a minface is determined only by the orientations of two null curves $\varphi$ and $\psi$. In summary, we have obtained the following theorem, which also gives a construction method of conformal curvature line coordinate systems and conformal asymptotic coordinate systems by using pseudo-arclength parameters. 
\begin{theorem}\label{theorem:p-Weierstrass_real}
Away from flat points, each minface $f : \Sigma \rightarrow \mathbb{L}^3$ can be written locally as the equation (\ref{eq:p-Weierstrass_real}). The Gaussian curvature $K$ is positive (resp. negative) if and only if $\varphi$ and $\psi$ have the same orientation (resp. different orientations). In this case, the Lorentz isothermal coordinate system $(x, y)=(\frac{u-v}{2}, \frac{u+v}{2})$ associated to the null coordinate system $(u, v)$ in (\ref{eq:p-Weierstrass_real}) is a conformal asymptotic (resp. conformal curvature line) coordinate system.
\end{theorem}


\begin{remark}\label{Milnor's study}
In Remark 1 in \cite{Milnor2}, 
Milnor normalized null coordinates $u, v$ so that $u$ and $v$ are \emph{Euclidean} arclength parameters of $\varphi/2$ and $\psi/2$ in the equation (\ref{null curves decomposition}), that is, on this coordinate system a timelike minimal surface $f$ can be written as 
\begin{align*}
f(u,v)&=\frac{1}{\sqrt{2}}\left(u-u_0, \int^u_{u_0}\cos A(\tau)d\tau, \int^u_{u_0}\sin A(\tau)d\tau \right) \nonumber\\
&+\frac{1}{\sqrt{2}}\left(v-v_0, \int^v_{v_0}\cos B(\tau)d\tau, \int^v_{v_0}\sin B(\tau)d\tau \right)+f(u_0,v_0),
\end{align*}
where, $A$ and $B$ are called the {\it Weierstrass functions}. By using these functions, Milnor gave the following formula giving control over the sign of the Gaussian curvature $K$:
\begin{equation}\label{eq:Milnor}
\mathrm{sgn}K=\mathrm{sgn}(A'B')
\end{equation}
After a straightforward calculation, we get $\det(\varphi', \varphi'', \varphi''')=(A')^3$, and hence $\mathrm{sgn}(A'B')=\varepsilon_\varphi \varepsilon_\psi$. About the sign of the Gaussian curvature, the equation (\ref{eq:Gauss}) is nothing but (\ref{eq:Milnor}). 
\end{remark}

\section{Behavior of the Gaussian curvature near singular points}\label{Sec.4}
In this section we investigate the behavior of the Gaussian curvature near non-degenerate singular points on a minface by using some notions about null curves given in Section \ref{subsection:null_curves} and results for the Gaussian curvature near regular points given in Section \ref{Sec.3}. 
\subsection{Frontals and fronts}
First we recall the singularity theory of frontals and fronts, see \cite{A, FSUY, SUY09, UY} for details. Let $U$ be a domain in $\mathbb{R}^2$ and $u$, $v$ are local coordinates on $U$. A smooth map $f$ : $U \longrightarrow \mathbb{R}^3$ is called a $frontal$ if there exists a unit vector field $n$ on $U$ such that $n$ is perpendicular to $df(TU)$ with respect to the Euclidean metric $\langle \  ,\  \rangle_E$ of $\mathbb{R}^3$. We call $n$ the {\it unit normal vector field} of a frontal $f$. Moreover if the {\it Legendrian lift} $L$ of a frontal $f$ 
\begin{equation*}
L=(f,n): U \longrightarrow \mathbb{R}^3\times S^2
\end{equation*}
is an immersion, $f$ is called a $front$. A point $p\in U$ where $f$ is not an immersion is called a {\it singular point} of the frontal $f$, and we call the set of singular points of $f$ the {\it singular set}. We can take the following smooth function $\lambda$ on $U$ 
\begin{equation*}
\lambda = \det(f_u, f_v, n)=\langle f_u \times_E f_v, n \rangle_E,
\end{equation*}
where $\times_E$ is the Euclidean vector product of $\mathbb{R}^3$. The function $\lambda$ is called the {\it signed area density function} of the frontal $f$. A singular point $p$ is called {\it non-degenerate} if $d\lambda_p \neq0$. 
The set of singular points of the frontal $f$ corresponds to zeros of $\lambda$. Let us assume that $p$ is a non-degenerate singular point of a frontal $f$, then there exists a regular curve $\gamma=\gamma(t): (-\varepsilon, \varepsilon)\longrightarrow U$ such that $\gamma(0)=p$ and the image of $\gamma$ coincides with the singular set of $f$ around $p$. We call $\gamma$ the {\it singular curve} and the direction of $\gamma'=\frac{d\gamma}{dt}$ the {\it singular direction}. On the other hand, there exists a non-zero vector $\eta \in \mathrm{Ker}(df_p)$ because $p$ is non-degenerate. We call $\eta$ the {\it null direction}.

Let $U_i$, $i=1, 2$ be domains of $\mathbb{R}^2$ and $p_i$, $i=1, 2$ be points in $U_i$. Two smooth maps $f_1:U_1\longrightarrow \mathbb{R}^3$ and $f_2:U_2\longrightarrow\mathbb{R}^3$ are {\it $\mathcal{A}$-equivalent} (or {\it right-left equivalent}) at the points $p_1 \in U_1$ and $p_2 \in U_2$ if there exist local diffeomorphisms $\Phi$ of $\mathbb{R}^2$ with $\Phi(p_1)=p_2$ and $\Psi$ of $\mathbb{R}^3$ with $\Psi(f_1(p_1))=f_2(p_2)$ such that $f_2 = \Psi \circ f_1\circ \Phi^{-1}$. A singular point $p$ of a map $f:U\longrightarrow \mathbb{R}^3$ is called a {\it cuspidal edge}, {\it swallowtail} or {\it cuspidal cross cap} if the map $f$ at $p$ is $\mathcal{A}$-equivalent to the following map $f_C$, $f_S$ or $f_{CCR}$ at the origin, respectively (see Figure \ref{Fig3}):
\begin{center}
$f_C(u,v)=(u^2,u^3,v)$, $f_S(u,v)=(3u^4+u^2v,4u^3+2uv,v)$, $f_{CCR}(u,v)=(u,v^2,uv^3)$.
\end{center}
Cuspidal edges and swallowtails are non-degenerate singular points of fronts, and these two types of singular points are generic singularities of fronts (cf. \cite{AGV}). In addition to these singular points, cuspidal cross caps often appear on minfaces, which are not singular points of fronts but are non-degenerate singular points of frontals.
\vspace{-0.2cm}

\begin{figure}[!h]
\begin{center}
\begin{tabular}{c}
\begin{minipage}{0.4\hsize}
\begin{center}
\vspace{-1.7cm}
\includegraphics[clip,scale=0.30,bb=0 0 500 409]{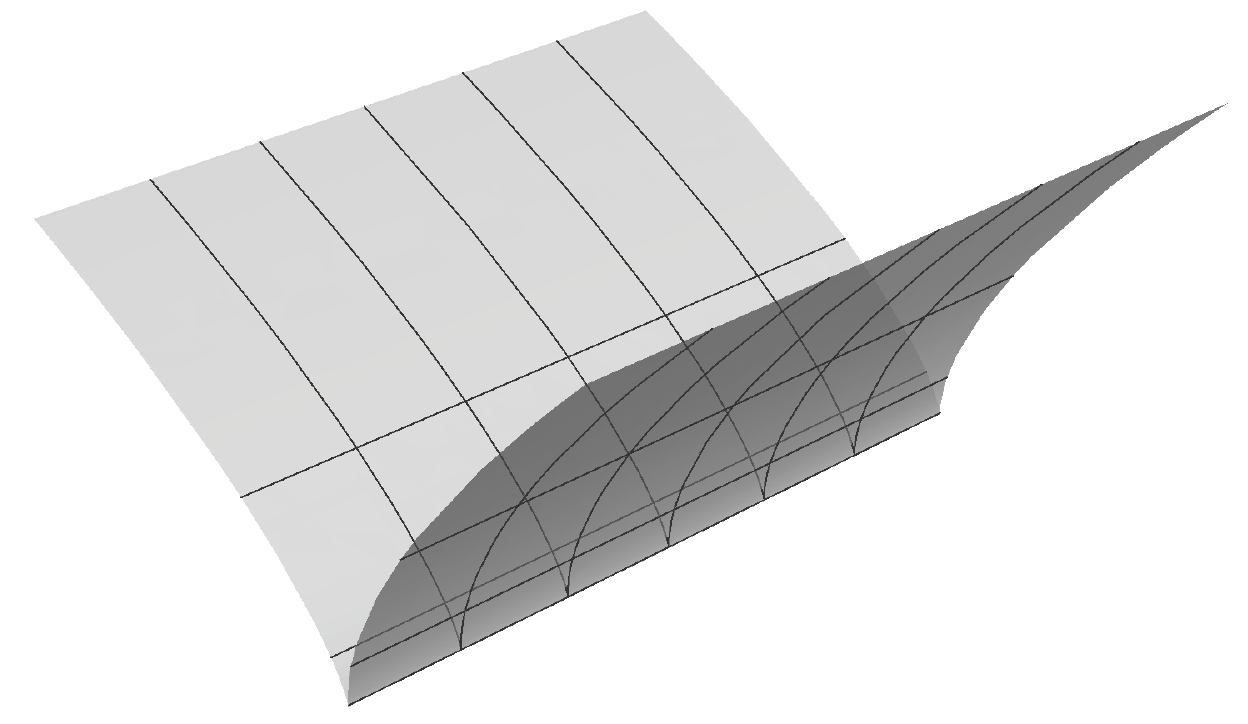}
\vspace{0.5cm}
\end{center}
\end{minipage}
\hspace{-1.7cm}
\begin{minipage}{0.4\hsize}
\begin{center}
\vspace{-3.0cm}
\includegraphics[clip,scale=0.30,bb=0 0 455 459]{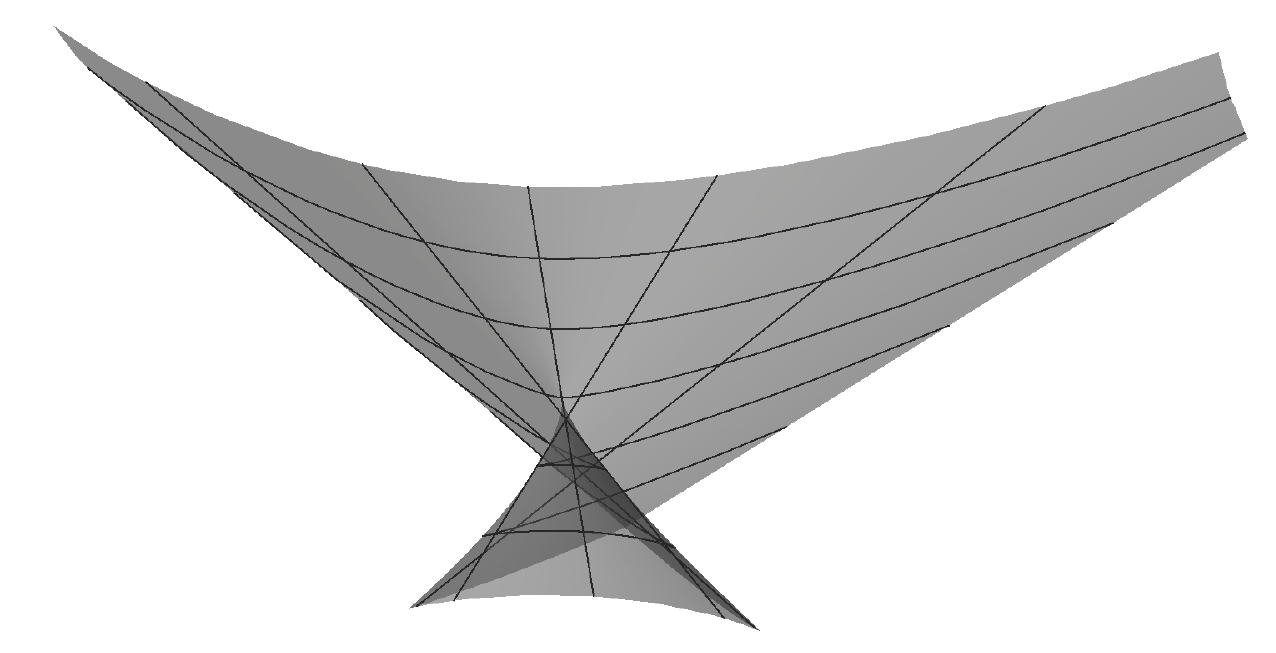}
\vspace{0cm}
\hspace{-1.1cm}
\end{center}
\end{minipage}
\hspace{-0.5cm}
\begin{minipage}{0.4\hsize}
\begin{center}
\vspace{-1.8cm}
\includegraphics[clip,scale=0.30,bb=0 0 455 449]{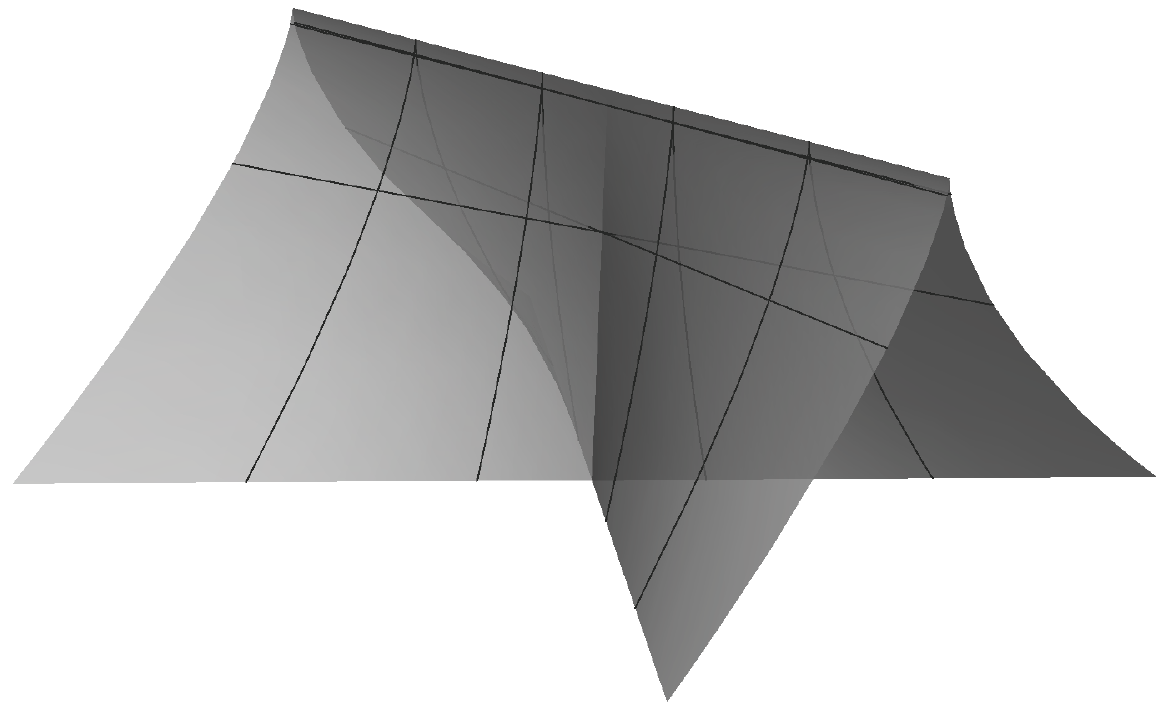}
\vspace{0.3cm}
\end{center}
\end{minipage}

\end{tabular}
\end{center}
\vspace{-0.2cm}
\caption{The cuspidal edge, swallowtail and cuspidal cross cap.}\label{Fig3}

\end{figure}

\subsection{Singular points on minfaces}
In \cite{FSUY}, Fujimori, Saji, Umehara and Yamada proved that the singular points of spacelike maximal surfaces in $\mathbb{L}^3$ generically consist of cuspidal edges, swallowtails and cuspidal cross caps. Similarly, these singular points frequently appear on timelike minimal surfaces. By using Facts as mentioned above, Takahashi gave the following criteria for cuspidal edges, swallowtails and cuspidal cross caps of minfaces by using their real Weierstrass data $(g_1, g_2,\omega_1, \omega_2)$. Now, we identify the Lorentz-Minkowski space $\mathbb{L}^3$ with the affine space $\mathbb{R}^3$.

\begin{fact}[\cite{T}]\label{Fact: T-lemma}
Let $f:U \longrightarrow \mathbb{L}^3$ be a minface and $p\in U$ a singular point. If we take the real Weierstrass data $(g_1, g_2,\hat{\omega}_1du, \hat{\omega}_2dv)$ on $U$, then $f$ is $\mathcal{A}$-equivalent to 
\begin{itemize}
\item[(i)] a cuspidal edge at $p$ if and only if $$\frac{g_1'}{g_1^2\hat{\omega}_1}-\frac{g_2'}{g_2^2\hat{\omega}_2}\neq 0\text{ and }\frac{g_1'}{g_1^2\hat{\omega}_1}+\frac{g_2'}{g_2^2\hat{\omega}_2}\neq0\text{ at } p,$$
\item[(ii)] a swallowtail at $p$ if and only if $$\frac{g_1'}{g_1^2\hat{\omega}_1}-\frac{g_2'}{g_2^2\hat{\omega}_2}\neq 0,\ \frac{g_1'}{g_1^2\hat{\omega}_1}+\frac{g_2'}{g_2^2\hat{\omega}_2}=0,\ \text{and}\ \left(\frac{g_1'}{g_1^2\hat{\omega}_1}\right)'\frac{g_2'}{g_2}-\left(\frac{g_2'}{g_2^2\hat{\omega}_2}\right)'\frac{g_1'}{g_1}\neq 0\text{ at }p,$$

\item[(iii)] a cuspidal cross cap at $p$ if and only if $$\frac{g_1'}{g_1^2\hat{\omega}_1}-\frac{g_2'}{g_2^2\hat{\omega}_2}=0,\ \frac{g_1'}{g_1^2\hat{\omega}_1}+\frac{g_2'}{g_2^2\hat{\omega}_2}\neq0,\ \text{and}\ \left(\frac{g_1'}{g_1^2\hat{\omega}_1}\right)'\frac{g_2'}{g_2}+\left(\frac{g_2'}{g_2^2\hat{\omega}_2}\right)'\frac{g_1'}{g_1}\neq0\text{ at }p.$$

\end{itemize}
\end{fact}

To prove Fact \ref{Fact: T-lemma}, Takahashi used the following fact. We shall recall the proof in \cite{T} which will be helpful to prove our main results. 
\begin{fact}[\cite{T}]\label{Lemma T-2}
Let $f:U \longrightarrow \mathbb{L}^3$ be a minface with the real Weierstrass data $(g_1, g_2,\hat{\omega}_1du, \hat{\omega}_2dv)$. Then
\begin{itemize}
\item[(i)]  a point $p$ is a singular point of $f$ if and only if $g_1(p)g_2(p)=1$.
\item[(ii)] $f$ is a frontal at any singular point $p$.
\item[(iii)] $f$ is a front at a singular point $p$ if and only if $\frac{g_1'}{g_1^2\hat{\omega}_1}-\frac{g_2'}{g_2^2\hat{\omega}_2}\neq 0$ at $p$. Moreover in this case, $p$ is automatically a non-degenerate singular point.
\end{itemize}
\end{fact}

\begin{proof}
Let $u$, $v$ be local coordinates on $U$. Since 
\begin{equation*}
f_u=\frac{\hat{\omega}_1}{2}(-1-g_1^2, 1-g_1^2, 2g_1),\quad f_v=\frac{\hat{\omega}_2}{2}(1+g_2^2, 1-g_2^2, -2g_2),
\end{equation*}
it holds that
\begin{equation*}
f_u\times f_v =\frac{\hat{\omega}_1\hat{\omega}_2}{2}(1-g_1g_2)(-g_1-g_2,g_1-g_2,-1-g_1g_2),
\end{equation*}
where $\times$ denotes the Euclidean outer product. Since $f$ is a minface, we obtain $\hat{\omega}_1\neq 0$ and $\hat{\omega}_2\neq 0$ at any point, and hence $p$ is a singular point if and only if $g_1(p)g_2(p)=1$. Moreover $f$ is a frontal with unit normal vector field
\begin{equation*}
n=\frac{1}{\sqrt{(1-g_1g_2)^2+2(g_1+g_2)^2}}(-g_1-g_2,g_1-g_2,-1-g_1g_2).
\end{equation*}

Next we prove (iii). Since $df_p$ and $dn_p$ are written as
\begin{equation*}
df_p=\frac{-g_1\omega_1+g_2\omega_2}{2}(g_1+g_2,g_1-g_2,-2),\quad dn_p=\frac{\left(-\frac{dg_1}{g_1}+\frac{dg_2}{g_2}\right)}{(g_1+g_2)\sqrt{2(g_1+g_2)^2}}(0,2,g_1-g_2),
\end{equation*}
where $\omega_1=\hat{\omega}_1du$ and $\omega_2=\hat{\omega}_2dv$, we obtain the following two vector fields $\eta$ and $\mu$ such that $df_p(\nu)=0$ and $dn_p(\mu)=0$:
\begin{equation}\label{eq:eta_nu}
\eta=\frac{1}{g_1\hat{\omega}_1}\left(\frac{\partial}{\partial u}\right)_p+\frac{1}{g_2\hat{\omega}_2}\left(\frac{\partial}{\partial v}\right)_p,\quad \mu=\frac{g_2'}{g_2}\left(\frac{\partial}{\partial u}\right)_p+\frac{g_1'}{g_1}\left(\frac{\partial}{\partial v}\right)_p.
\end{equation}
On the other hand, the minface $f$ is a front at $p$  if and only if the directions $\eta$ and $\mu$ are linearly independent (see, for example, proof of Lemma 3.3 in \cite{UY}) and by the equations (\ref{eq:eta_nu}) we get
\begin{equation*}
\det(\eta, \mu)=\frac{g_1'}{g_1^2\hat{\omega}_1}-\frac{g_2'}{g_2^2\hat{\omega}_2} {\text \ at\ }p, 
\end{equation*}
which proves the first part of the conclusion. Moreover, the signed area density function $\lambda$ can be written as 
\begin{equation*}
\lambda=-\frac{\hat{\omega}_1\hat{\omega}_2}{2}(1-g_1g_2)\sqrt{(1-g_1g_2)^2+2(g_1+g_2)^2},
\end{equation*}
and hence its derivative at $p$ can be written as follows:
\begin{equation}\label{eq:dlambda}
d\lambda_p=\frac{\hat{\omega}_1\hat{\omega}_2}{\sqrt{2}}|g_1+g_2|\left(\frac{dg_1}{g_1}+\frac{dg_2}{g_2}\right).
\end{equation}
Therefore a singular point $p$ is non-degenerate if and only if $g_1(p)\neq0$ or $g_2(p)\neq0$, and hence we have proved the desired result.
\end{proof}

\subsection{Behavior of the Gaussian curvature near singular points}
Now we are in the position to investigate the behavior of the Gaussian curvature near singular points of minfaces by using the facts given above. 

\begin{proof}[Proof of Theorem \ref{Maintheorem}]
We use the representation (\ref{eq:p-Weierstrass_real1}) for a minface $f$. Since $\varphi$ in (\ref{eq:p-Weierstrass_real1}) satisfies $\langle \varphi'', \varphi'' \rangle =4{g_1'}^2\hat{\omega}_1^2$ and $\hat{\omega}_1\neq 0$ on the minface $f$, $\varphi$ is degenerate at $p$ if and only if $g_1'(p)=0$. Similarly, $\psi$ is degenerate at $p$ if and only if $g_2'(p)=0$. By (i) of Fact \ref{Fact: T-lemma}, near a cuspidal edge $g_1'\neq 0$ or  $g_2'\neq 0$. Hence there is no umbilic point near $p$ by Proposition \ref{prop:flat points}. Next we prove (ii). If we assume that one of $g_1'(p)$ and $g_2'(p)$ vanishes, then the other one also vanishes by (i) of Fact \ref{Fact: T-lemma}. By (iii) of Fact \ref{Lemma T-2}, it contradicts  the assumption that $f$ is a front, that is, there is no flat point near $p$. By Proposition \ref{prop:flat points} and Lemma \ref{lemma:non-deg}, we can take pseudo-arclength parameters of $\varphi$ and $\psi$, that is, 
\begin{equation}\label{eq: pseudo-arc}
g_1'\hat{\omega}_1(u)=-\frac{\varepsilon_\varphi}{2}, \quad g_2'\hat{\omega}_2(v)=-\frac{\varepsilon_\psi}{2},
\end{equation} 
and hence 
\begin{equation}\label{eq: two discriminants}
\frac{g_1'}{g_1^2\hat{\omega}_1}-\frac{g_2'}{g_2^2\hat{\omega}_2}=-\frac{\varepsilon_\varphi}{2g_1^2{\hat{\omega}_1}^2}+\frac{\varepsilon_\psi}{2g_2^2{\hat{\omega}_2}^2} \text{\ and\ }  \frac{g_1'}{g_1^2\hat{\omega}_1}+\frac{g_2'}{g_2^2\hat{\omega}_2}=-\left(\frac{\varepsilon_\varphi}{2g_1^2{\hat{\omega}_1}^2}+\frac{\varepsilon_\psi}{2g_2^2{\hat{\omega}_2}^2}\right).
\end{equation}
Since $f$ is a front at $p$, the quantity $\frac{g_1'}{g_1^2\hat{\omega}_1}-\frac{g_2'}{g_2^2\hat{\omega}_2}$ does not vanish at $p$ by (iii) of Fact \ref{Lemma T-2}. On the other hand, if we assume that the singular point $p$ is not a cuspidal edge, then the quantity $\frac{g_1'}{g_1^2\hat{\omega}_1}+\frac{g_2'}{g_2^2\hat{\omega}_2}$ vanishes at $p$ by (i) of Fact \ref{Fact: T-lemma}. Therefore by the second equation of (\ref{eq: two discriminants}), the orientations of $\varphi$ and $\psi$ are different. Hence, by Theorem \ref{theorem:p-Weierstrass_real}, the Gaussian curvature $K$ is negative and $K$ diverges to $-\infty$ at $p$. Finally if we assume that $f$ is not a front at $p$ and $p$ is a non-degenerate singular point, then the quantity $\frac{g_1'}{g_1^2\hat{\omega}_1}-\frac{g_2'}{g_2^2\hat{\omega}_2}$ vanishes at $p$. Hence, if one of $g_1'(p)$ and $g_2'(p)$ vanishes, then the other one also vanishes, which contradicts the assumption that $p$ is non-degenerate and the equation (\ref{eq:dlambda}). Therefore, there is no flat point near $p$. By taking pseudo-arclength parameters of $\varphi$ and $\psi$ with (\ref{eq: pseudo-arc}) again and considering the first equation of (\ref{eq: two discriminants}), we conclude that the orientations of $\varphi$ and $\psi$ are the same. By Theorem \ref{theorem:p-Weierstrass_real}, the Gaussian curvature $K$ is positive and $K$ diverges to $\infty$ at $p$, which completes the proof.
\end{proof}

\begin{remark}\label{sign of CE}
In general, the sign of the Gaussian curvature near cuspidal edges of a minface cannot be determined. If we take the real Weierstrass data
\begin{center}
$g_1(u)=u$, $g_2(v)=1+v^2$, $\omega_1(u)=du$ and $\omega_2(v)=dv$, 
\end{center}
in the equation (\ref{eq:p-Weierstrass_real1}),
then the singular set $\Sigma_f$ is determined by the equation $g_1(u)g_2(v)=u(1+v^2)=1$, that is, $\Sigma_f=\{(\frac{1}{1+v^2},v)\in \mathbb{R}^2\mid v \in \mathbb{R}\}$ and quantities in (i) of Fact \ref{Fact: T-lemma} are computed as follows
\begin{center}
$\frac{g_1'}{g_1^2\hat{\omega}_1}\pm \frac{g_2'}{g_2^2\hat{\omega}_2}=\frac{1}{u^2}\pm\frac{2v}{(1+v^2)^2}=\frac{(1+v^2)^4\pm2v}{(1+v^2)^2}=\frac{(1\pm v)^2+3v^2+6v^4+4v^6+v^8}{(1+v^2)^2}>0$ on  $\Sigma_f$.
\end{center}
Therefore, the set of singular points $\Sigma_f$ consists of only cuspidal edges. On the other hand, the Gaussian curvature is $K(u,v)=\frac{2v}{(u(1+v^2)-1)^4}$. Hence the sign of the Gaussian curvature cannot be determined near cuspidal edges in general. Moreover, the Gaussian curvature of this example does not diverge along the quasi-umbilic curve $v=0$ (the curve appears as the boundary of black and gray parts in Figure \ref{Fig4}).
\begin{figure}[htbp]
 \vspace{-0.45cm}
\begin{center}
\includegraphics[clip,scale=0.30,bb=0 0 310 300]{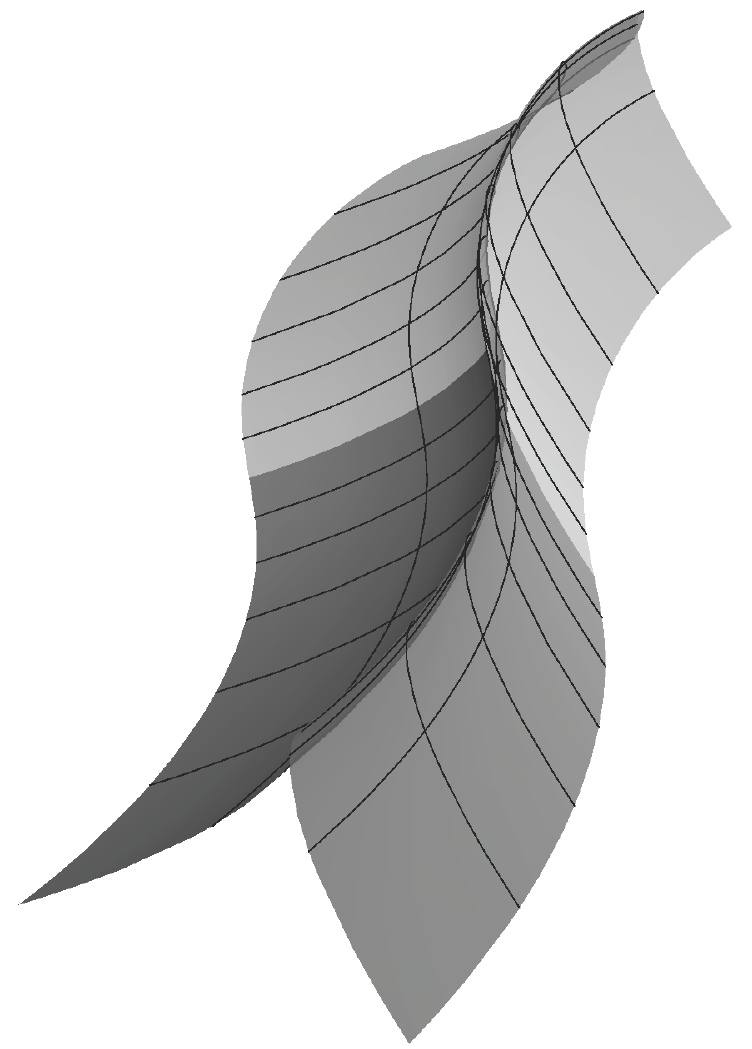}
 \vspace{0cm}
\caption{A minface with cuspidal edges on which the sign of the Gaussian curvature changes along a quasi-umbilic curve.}\label{Fig4}
\end{center}
\end{figure}
\end{remark}

\begin{remark}\label{maxfaces}In \cite{UY}, Umehara and Yamada introduced the notion of maxfaces in $\mathbb{L}^3$, and proved that any maxface $f$ is locally represented as 
\begin{equation*}
f(z)=\Re\int^z_{z_0}(-2G,1+G^2, i(1-G^2))\eta,
\end{equation*}
where $(G, \eta)$ is a pair of a meromorphic function and a holomorphic $1$-form on a simply connected domain in $\mathbb{C}$ containing a base point $z_0$ such that $(1+|G|^2)^2|\eta|^2\neq 0$ on the domain. Moreover, the first fundamental form of $f$ is given by $\mathrm{I}=(1-|G|^2)^2|\eta|^2$, and hence a point $z$ is a singular point of $f$ if and only if $|G(z)|=1$.  
By using $(G, \eta)$, the intrinsic Gaussian curvature $K$ of $f$ can be written as
\begin{equation}\label{eq:Gauss curvature on maxfaces}
K = \cfrac{4|dG|^2}{(1-|G|^2)^4|\eta|^2},
\end{equation}
where the non-degeneracy of a singular point $p$ means $dG_p\neq 0$ (Lemma 3.3 in \cite{UY}). Therefore at a non-degenerate singular point $p$ of any maxface, the Gaussian curvature $K$ always diverges to $\infty$.
\end{remark}

\begin{example}[\cite{IT, Konderak,T}]\label{Lorentzian Enneper}
If we take the real Weierstrass data
\begin{center}
$g_1(u)=u$, $g_2(v)=-v$, $\omega_1(u)=\frac{1}{2}du$ and $\omega_2(v)=\frac{1}{2}dv$
\end{center}
in the equation (\ref{eq:p-Weierstrass_real1}), we obtain the following two null curves
\begin{equation*}
\varphi(u)=\frac{1}{2}(-u-\frac{u^3}{3},u-\frac{u^3}{3},u^2) \text{ and } \psi(v)=\frac{1}{2}(v+\frac{v^3}{3},v-\frac{v^3}{3},v^2).
\end{equation*}
The surface obtained by these two null curves is called the {\it timelike Enneper surface of isothermic type} or an {\it analogue of Enneper's surface}. Since $2g_1'\hat{\omega}_1=1$ and $2g_2'\hat{\omega}_2=-1$, $\varphi$ and $\psi$ are parametrized by pseudo-arclength parameters and have negative and positive orientations, respectively. Hence, Theorem \ref{theorem:p-Weierstrass_real} states that the Gaussian curvature $K$ is negative. Moreover the singular set is $\Sigma_f=\{(u,v)\in \mathbb{R}^2\mid uv=-1\}$ and the quantities in Fact \ref{Fact: T-lemma} are computed as $\frac{g_1'}{g_1^2\hat{\omega}_1}\pm \frac{g_2'}{g_2^2\hat{\omega}_2}=2(v^2\pm u^2)$ on $\Sigma_f$. Therefore, $\Sigma_f$ consists of cuspidal edges $\Sigma_f \setminus \{(1,-1), (-1,1)\}$ and swallowtails $\{(1,-1), (-1,1)\}$. By (ii) of Fact \ref{duality} in Appendix \ref{Appendix:A}, the conjugate minface of the timelike Enneper surface of isothermic type $f^*$ defined by $(\ref{eq:conjugate})$ has cuspidal edges $\Sigma_f \setminus \{(1,-1), (-1,1)\}$ and cuspidal cross caps $\{(1,-1), (-1,1)\}$, see Figure \ref{Fig5}.
\vspace{+0.1cm}
\begin{figure}[!h]
\begin{center}
\begin{tabular}{c}
\hspace{+2.2cm}
\begin{minipage}{0.4\hsize}
\begin{center}
\vspace{-1.6cm}
\includegraphics[clip,scale=0.28,bb=0 0 500 509]{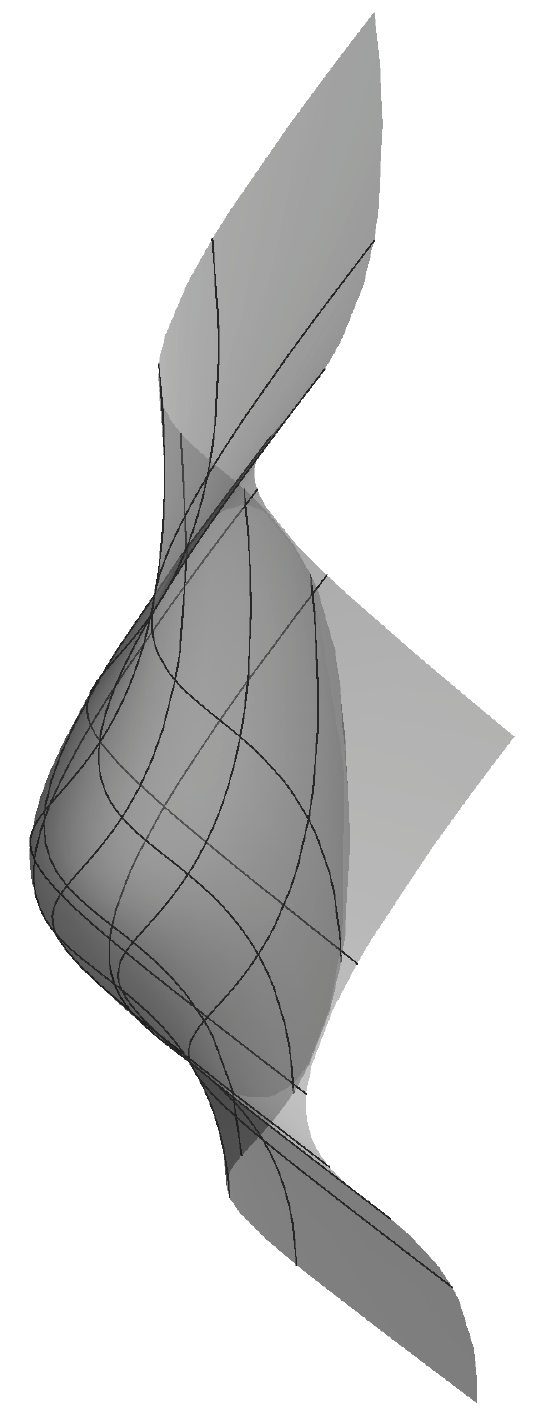}
\vspace{0.5cm}
\end{center}
\end{minipage}
\begin{minipage}{0.4\hsize}
\begin{center}
\vspace{-1.1cm}
\includegraphics[clip,scale=0.27,bb=0 0 555 449]{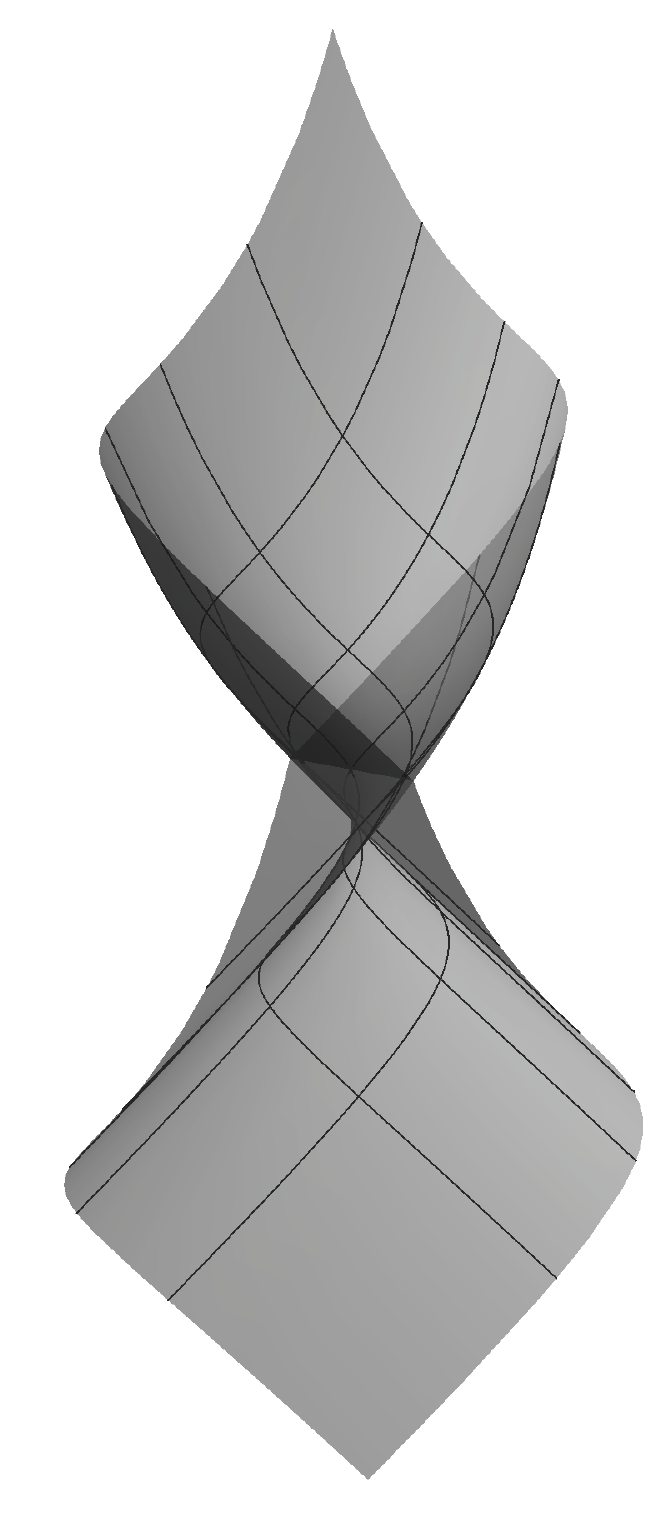}
\vspace{0.1cm}
\end{center}
\end{minipage}

\end{tabular}
\end{center}
\vspace{-0.5cm}
\caption{The timelike Enneper surface of isothermic-type and its conjugate.}\label{Fig5}

\end{figure}

\end{example}
As Example \ref{ex:K-change} umbilic points and quasi-umbilic points on a timelike minimal surface are not isolated in general. We also saw an example in Remark \ref{sign of CE} where a curve of quasi-umbilic points on a minface accumulates to a cuspidal edge. As a corollary of Theorem \ref{Maintheorem}, we obtain the following:

\begin{corollary}\label{Cor: accumulate points}
The following (i) and (ii) hold:
\begin{itemize}
\item[(i)] Umbilic points do not accumulate to a non-degenerate singular point of a minface.
\item[(ii)] If quasi-umbilic points accumulate to a non-degenerate singular point $p$ of a minface, then $p$ is a cuspidal edge.
\end{itemize}
\end{corollary}
\begin{proof}
The claim (i) follows from the equation(\ref{eq:dlambda}) and (i) of Proposition \ref{prop:flat points}. To prove the claim (ii), if we assume that quasi-umbilic points accumulate to a non-degenerate singular point $p$ which is not a cuspidal edge, then the condition $K(p)=0$ contradicts to (ii) or (iii) of Theorem \ref{Maintheorem}.
\end{proof}
\begin{remark}
In contrast with the corollary as above, umbilic points on a maxface do not accumulate to a non-degenerate singular point by the equation (\ref{eq:Gauss curvature on maxfaces}) and the non-degeneracy of a singular point in Remark \ref{maxfaces}.
\end{remark}

In the end of this section, we give a criterion for the sign of the Gaussian curvature near cuspidal edges on minfaces. Let $f:U \longrightarrow \mathbb{R}^3$ be a front with the unit normal vector field $n$ and $\gamma=\gamma(t)$ a singular curve on $U$ consists of cuspidal edges. By (i) of Fact \ref{Fact: KRSUY}, $\gamma$ is a regular curve and we can take the null vector fields $\eta$ such that $(\gamma'(t), \eta(t))$ is positively oriented with respect to the orientation of $U$.
The {\it singular curvature} $\kappa_s$ of the cuspidal edge $\gamma$ was defined in \cite{SUY09} as
\begin{equation*}
\kappa_s(t)=\mathrm{sgn}(d\lambda(\eta))\frac{\det(\hat{\gamma}'(t),\hat{\gamma}''(t),n)}{|\hat{\gamma}'(t)|^3},
\end{equation*}
where $\hat{\gamma}=f\circ \gamma$ and $|\hat{\gamma}'(t)|=\langle\hat{\gamma}'(t),\hat{\gamma}'(t) \rangle_E^{1/2}$. The singular curvature is an intrinsic invariant of cuspidal edges, and related to the behavior of the Gaussian curvature as stated in Introduction. For minfaces, the singular curvature characterizes the sign of the Gaussian curvature near cuspidal edges:

\begin{theorem}\label{Theorem: Gaussian curvature near cusps}
Let $f:U \longrightarrow \mathbb{L}^3$ be a minface with the real Weierstrass data $(g_1, g_2,\hat{\omega}_1du, \hat{\omega}_2dv)$ and $\gamma(t)$ the singular curve passing through a cuspidal edge $p=\gamma(0)$. Then the Gaussian curvature $K$ and the singular curvature $\kappa_s$ have the same sign. In particular, zeros of $\kappa_s$ correspond to either zeros of $g_1'$ or $g_2'$. 
\end{theorem}

\begin{proof}
By the proofs of Facts \ref{Fact: T-lemma} (see Appendix A) and \ref{Lemma T-2}, we can compute
\begin{align}
\det(\hat{\gamma}'(t),\hat{\gamma}''(t),n)=\hat{\omega}_1^2\hat{\omega}_2^2g_1'g_2'\sqrt{\frac{(g_1+g_2)^2}{2}} \left(\frac{g_1'}{g_1^2\hat{\omega}_1}+\frac{g_2'}{g_2^2\hat{\omega}_2}\right)^2, \nonumber \\
|\hat{\gamma}'|=\sqrt{\frac{\hat{\omega}_1^2\hat{\omega}_2^2(g_1+g_2)^2}{2}}\left(\frac{g_1'}{g_1^2\hat{\omega}_1}+\frac{g_2'}{g_2^2\hat{\omega}_2}\right)\quad \text{and}\quad \mathrm{sgn}(d\lambda(\eta))=\mathrm{sgn}(\hat{\omega}_1\hat{\omega}_2),\nonumber
\end{align}
where we take the null vector field $\eta$ satisfying the condition
\begin{equation*}
\det(\gamma',\eta)=\left(\frac{g_1'}{g_1^2\hat{\omega}_1}+\frac{g_2'}{g_2^2\hat{\omega}_2}\right) >0.
\end{equation*}
Therefore, the singular curvature $\kappa_s$ is written as 
\begin{equation}\label{eq:singular_curvature}
\kappa_s=\frac{2g_1'g_2'}{\hat{\omega}_1\hat{\omega}_2(g_1+g_2)^2}\frac{1}{\left(\frac{g_1'}{g_1^2\hat{\omega}_1}+\frac{g_2'}{g_2^2\hat{\omega}_2}\right)}.
\end{equation}
Hence, zeros of $\kappa_s$ correspond to either zeros of $g_1'$ or $g_2'$. On the other hand, the Gaussian curvature $K$ of the minface $f$ which is represented as (\ref{eq:p-Weierstrass_real1}) is written as
\begin{equation}\label{eq:Gaussian_curvature}
K=\frac{4g_1'g_2'}{\hat{\omega}_1\hat{\omega}_2(1-g_1g_2)^4}.
\end{equation}
By (\ref{eq:singular_curvature}) and (\ref{eq:Gaussian_curvature}), we obtain the desired result.
\end{proof}


  \renewcommand{\theequation}{A.\arabic{equation}}
\setcounter{equation}{0}
\appendix
\section{Geometry of minfaces}\label{Appendix:A}
In this appendix we give a precise description of the notion of minfaces and their representation formulas based on the work by Takahashi \cite{T}.

First we shall recall the notion of paracomplex algebra. For a more detailed exposition on paracomplex numbers, see \cite{CFG,IT,Konderak} and their references. Let $\mathbb{C}'$ be the $2$-dimensional commutative algebra of the form $\mathbb{C}'=\mathbb{R}1\oplus \mathbb{R} j$ with multiplication law:
\begin{equation*}
j\cdot 1=1\cdot j=j,\quad j^2=1.
\end{equation*}
An element of $\mathbb{C}'$ is called a {\it paracomplex number} and $\mathbb{C}'$ is called the {\it paracomplex algebra}. Some authors use the terminology {\it split-complex numbers} or {\it Lorentz numbers} instead of paracomplex numbers. For a paracomplex number $z=x+jy$, we call $\Re{z}:=x$, $\Im{z}:=y$ and $\bar{z}:=x-jy$ the real part, the imaginary part and the conjugate of $z$, respectively. The paracomplex algebra $\mathbb{C}'$ can be identified with the Minkowski plane $\mathbb{L}^2=(\mathbb{R}^2, \langle \ ,\ \rangle _{\mathbb{L}^2}=-dx^2+dy^2)$ as follows:
\begin{equation*}
\mathbb{C}' \ni z=x+jy \longleftrightarrow \bm{z}=(x, y)\in \mathbb{L}^2.
\end{equation*}
Under the identification, the scalar product $\langle \bm{z_1}, \bm{z_2}\rangle_{\mathbb{L}^2}$ of $\mathbb{L}^2$ can be written as $-\Re{(\bar{z}_1z_2)}$. In particular, $\langle \bm{z},\bm{z}\rangle _{\mathbb{L}^2}=-z\bar{z}$ and we define $\langle z \rangle^2:=z\bar{z}$. We also define the {\it $n$-dimensional paracomplex space} as ${\mathbb{C}'}^n:=\{(z^0,z^1,\cdots,z^{n-1})\mid z^0, z^1,\cdots, z^{n-1}\in \mathbb{C}'\}$. 

A $(1,1)$-tensor field $J$ on a 2-dimensional oriented manifold $\Sigma$ is called an {\it almost paracomplex structure} if $J$ satisfies $J^2=\mathrm{id}$ and $\mathrm {dim}(V_{-})=\mathrm{dim}(V_{+})=1$, where $V_-$ and $V_+$ are $\pm 1$-eigenspaces for $J$. As pointed out in \cite{IT,W} every almost complex structure $J$ on $\Sigma$ is integrable, that is, there exists a coordinate system $(u,v)$ compatible with the orientation of $\Sigma$ such that $J(\frac{\partial}{\partial u})=\frac{\partial}{\partial u}$ and $J(\frac{\partial}{\partial v})=-\frac{\partial}{\partial v}$ near each point. We also call $(u,v)$ a {\it null coordinate system}, $(x=\frac{u-v}{2},y=\frac{u+v}{2})$ a {\it Lorentz isothermal coordinate system} on $(\Sigma, J)$ and $(\Sigma, J)$ a {\it para-Riemann surface}. 

A smooth map $\varphi$ between para-Riemann surfaces $(M, J)$ and $(N,J')$ is called {\it paraholomorphic} if $d\varphi \circ J=J'\circ d\varphi$. Paraholomorphicity of maps locally can be characterized as follows:
\begin{fact}[\cite{T}]\label{fact:paraholomorphicity}
Let $D\subset \mathbb{C}'$ be a domain with a coordinate $z=x+jy=\frac{u-v}{2}+j\frac{u+v}{2}$. A function $\varphi=\varphi^1+j\varphi^2$ is paraholomorphic if and only if there exist functions $f=f(u)$ and $g=g(v)$ such that $\varphi(z)=\frac{f(u)+g(v)}{2}+j\frac{f(u)-g(v)}{2}$.
\end{fact} It follows directly from the observations that $\varphi$ satisfies $d\varphi (J(\frac{\partial}{\partial u}))=J(d\varphi(\frac{\partial}{\partial u}))$ if and only if there exists a function $g=g(v)$ such that $\varphi^1-\varphi^2=g$, and $d\varphi (J(\frac{\partial}{\partial v}))=J(d\varphi(\frac{\partial}{\partial v}))$ if and only if there exists a function $f=f(u)$ such that $\varphi^1+\varphi^2=f$, where $J$ is the canonical paracomplex structure on $\mathbb{C}'$. A $1$-form $\omega$ is called {\it paraholomorphic} if $\omega$ can be written as $\omega=\hat{\omega}dz$ in any local paracomplex coordinate $z$ with a paraholomorphic function $\hat{\omega}$.

In \cite{T}, Takahashi introduced the notion of timelike minimal surfaces with some kind of singular points of rank one, which are called {\it minfaces} as follows: 
\begin{definition}[\cite{T}]\label{def:minface}
Let $(\Sigma,J)$ be a para-Riemann surface. A smooth map $f: \Sigma \longrightarrow \mathbb{L}^3$ is a {\it minface} if 
there is an open dense set $W\subset \Sigma$ such that $f$ is a conformal timelike minimal immersion on $W$, and
on each null coordinate system $(u,v)$, $f_u\neq 0$ and $f_v\neq 0$ at each point. A point $p\in \Sigma$ is called a {\it singular point} of $f$ if $f$ is not an immersion at $p$. \end{definition}

\begin{remark}
In \cite{KKSY}, Kim, Koh, Shin and Yang defined the notion of {\it generalized timelike minimal surfaces} as follows: Let $\Sigma$ be a $2$-dimensional $C^2$-manifold. A non-constant map $f: \Sigma \longrightarrow \mathbb{L}^3$ is called a {\it generalized timelike minimal surface} if at each point of $\Sigma$ there exists a local coordinate system $(x, y)$ such that 
(i) $\langle f_x, f_x \rangle \equiv -\langle f_y, f_y \rangle \geq0$, $\langle f_x, f_y \rangle \equiv0$,
(ii) $f_{xx}-f_{yy}\equiv 0$ and
(iii) $\langle f_x, f_x \rangle = -\langle f_y, f_y \rangle >0$ almost everywhere on $\Sigma$. A singular point of such a surface is in either $\mathcal{A}:=\{p \mid \text{$f_x$  or $f_y$ is lightlike}\}$ or $\mathcal{B}:=\{p \mid \text{$df_p$ vanishes}\}$. By definition, a minface is a generalized timelike minimal surface without singular points belonging to $\mathcal{B}$. However, the converse is not true, that is, we can construct an example of a generalized timelike minimal surface with only singular points belonging to $\mathcal{A}$ which is not a minface by taking only one of two generating null curves with a singular point.
\end{remark}

A paraholomorphic map $F=(F^0,F^1,F^2): \Sigma \longrightarrow {\mathbb{C}'}^3$ is called a {\it Lorentzian null map} if 
\begin{equation*}
F_z\cdot F_z:= -(F^0_z)^2+(F^1_z)^2+(F^2_z)^2\equiv 0
\end{equation*}
holds on $\Sigma$, where $z=\frac{u-v}{2}+j\frac{u+v}{2}$ is a local paracomplex coordinate in a domain $U\subset \Sigma$ and $F_z=\frac{\partial F}{\partial z}=\frac{1}{2}\left[\left( \frac{\partial F}{\partial u}-\frac{\partial F}{\partial v}\right)+j\left( \frac{\partial F}{\partial u}+\frac{\partial F}{\partial v}\right)\right]$. By the paraholomorphicity of $F$ and Fact \ref{fact:paraholomorphicity}, we can take the decomposition of $F: U \longrightarrow {\mathbb{C}'}^3$ as 
\begin{equation}\label{eq:F}
F(z)=\frac{\varphi(u)+\psi(v)}{2}+j\frac{\varphi(u)-\psi(v)}{2}.
\end{equation}
Since 
\begin{equation*}
F_z\cdot F_z=\frac{1}{2}\left[\langle \varphi'(u), \varphi'(u) \rangle+\langle \psi'(v), \psi'(v) \rangle+2j\left(\langle \varphi'(u), \varphi'(u) \rangle-\langle \psi'(v), \psi'(v) \rangle\right)\right],
\end{equation*}
we have the following:
\begin{fact}[\cite{T}]\label{fact:Lorentzian null map}
Let $U$ be a domain in $\mathbb{C}'$ and $z=x+jy=\frac{u-v}{2}+j\frac{u+v}{2}$ be the canonical coordinate on $\mathbb{C}'$. If we take the decomposition (\ref{eq:F}), then the following conditions are equivalent:
\begin{itemize}
\item[(i)] $F$ is a Lorentzian null map, 
\item[(ii)] $\varphi$ and $\psi$ satisfy $\langle \varphi'(u), \varphi'(u) \rangle =0$ and $\langle \psi'(v), \psi'(v) \rangle =0$.
\end{itemize}
\end{fact}

\begin{remark}\label{Immersed condition}
The condition (ii) above does not mean that $\varphi$ and $\psi$ are null curves because there may be a point $p$ such that $\varphi'(p)=0$ or $\psi'(p)=0$. Since the Jacobi matrix of $F$ can be written as 
\begin{align*}
JF= \frac{1}{2}\left(
    \begin{array}{cc}
      \varphi' &\psi'   \\
     \varphi'  &-\psi'
      \end{array}
  \right),
\end{align*}
a necessary and sufficient condition that the Lorentzian null map $F$ as above is an immersion is both of $\varphi$ and $\psi$ are null curves.
\end{remark}

Similar to the case of maxfaces (see Proposition 2.3 in \cite{UY}), any minface can be written by using a paraholomorphic Lorentzian null immersion as follows:
\begin{fact}[\cite{T}]\label{fact:p-lift}
Let $(\Sigma, J)$ be a para-Riemann surface and $f : \Sigma \rightarrow \mathbb{L}^3$ be a minface. Then there is a paraholomorphic Lorentzian null immersion $F: \widetilde{\Sigma} \longrightarrow {\mathbb{C}'}^3$ such that $f\circ \pi=F+\overline{F}$, where $\pi: \widetilde{\Sigma} \longrightarrow \Sigma$ is the universal covering map of $\Sigma$. Conversely, if $F: \widetilde{\Sigma} \longrightarrow {\mathbb{C}'}^3$ is a paraholomorphic Lorentzian null immersion which gives a timelike minimal immersion $f=F+\overline{F}$ on an open dense set, then $f$ is a minface.
\end{fact}

\begin{proof}
By Definition \ref{def:minface}, there exists an open dense set $W\subset \Sigma$ such that $f|_W$ is a conformal timelike minimal immersion. Then if we take a paracomplex coordinate $z$ in a domain $U$, we  obtain $f_{z\bar{z}}\equiv 0$ on $U\cap W$. Since $W$ is a dense set, the equality as above holds on $U$, and hence $\partial f=f_zdz$ is a paraholomorphic 1-form on $\Sigma$. We can take a paraholomorphic map $F: \tilde{\Sigma} \longrightarrow {\mathbb{C}'}^3$ such that $dF=\partial(f\circ\pi)$. Since $\partial(F+\bar{F})=dF$, there exists a real number $c$ such that $F+\bar{F}=f\circ\pi+c$. In particular, we can take $c=0$. Let us take null coordinates $u$, $v$ in a domain $U\subset \Sigma$ near any point $p\in \Sigma$, and consider null coordinates $\tilde{u}$, $\tilde{v}$ in each connected component of $\pi^{-1}(U)$ such that $\pi \circ \tilde{u}=u$ and $\pi \circ \tilde{v}=v$. By Fact \ref{fact:paraholomorphicity}, we can take the following decomposition
\begin{equation*}
F(\tilde{u}, \tilde{v})=\frac{\varphi(\tilde{u})+\psi(\tilde{v})}{2}+j\frac{\varphi(\tilde{u})-\psi(\tilde{v})}{2}.
\end{equation*}
Since $F+\bar{F}=f\circ\pi+c$, we obtain
\begin{equation*}
f_u(\pi(\tilde{u},\tilde{v}))=\varphi'(\tilde{u}),\quad f_v(\pi(\tilde{u},\tilde{v}))=\psi'(\tilde{v}).
\end{equation*}
By the assumption that $f|_W$ is a conformal timelike minimal immersion and Fact \ref{fact:Lorentzian null map}, $F$ is a Lorentzian null map on $\tilde{\Sigma}$. By Remark \ref{Immersed condition}, we conclude that $F$ is an immersion on $\tilde{\Sigma}$. Next we prove the converse. By the assumption, $f=F+\bar{F}$ satisfies the condition (i) of Definition \ref{def:minface}, and the condition (ii) of Definition \ref{def:minface} follows from Remark \ref{Immersed condition}.
\end{proof}

In particular, any minface $f$ can be written as the equation (\ref{null curves decomposition}). 
We call the paraholomorphic Lorentzian null immersion $F$ as above the {\it paraholomorphic lift} of the minface $f$. Moreover the following Weierstrass-type representation formula for minfaces is known.

\begin{fact}[Local version of the Weierstrass representation formula in {\cite{T}}]\label{theorem:p-Weierstrass}
Let $f : \Sigma \rightarrow \mathbb{L}^3$ be a minface.  For each point $p \in \Sigma$, after a rotation with respect to the time axis, there exist a paraholomorphic function $g$ and a paraholomorphic 1-form $\omega=\hat{\omega}dz$ which are defined near $p$ such that $f$ can be written as follows
\begin{equation}\label{eq:p-Weierstrass}
  f(z)
  =\Re \int^z_{z_0}\left( -1-g^2, j(1-g^2), 2g \right)\omega
    +f(z_0).
\end{equation}

Moreover if we decompose paraholomorphic functions $g$ and $\hat{\omega}$ into 
\begin{align}
g(z)&=\frac{g_1(u)+g_2(v)}{2}+j\frac{g_1(u)-g_2(v)}{2},\label{decomp1}\\ 
\hat{\omega}(z)&=\frac{\hat{\omega}_1(u)+\hat{\omega}_2(v)}{2}+j\frac{\hat{\omega}_1(u)-\hat{\omega}_2(v)}{2},\quad z=x+jy=\frac{u-v}{2}+j\frac{u+v}{2}, \label{decomp2}
\end{align}
then, a minface $f$ can be decomposed into two null curves as the equation (\ref{eq:p-Weierstrass_real1}).
\end{fact}

\begin{proof}
Let us take a paracomplex coordinate $z$ near $p$ and consider the followings
\begin{equation*}
\omega=\hat{\omega}dz=(-f^0_z+jf^1_z)dz, \quad g=\frac{f^2_z}{\hat{\omega}}.
\end{equation*}
Here, we shall prove that after a rotation with respect to the time axis, we can take $\langle \hat{\omega}(p)\rangle^2 \neq 0$, that is, $g$ is locally paraholomorphic. Let us assume that $\langle \hat{\omega}(p)\rangle^2=0$ and take the paraholomorphic lift $F$ of $f$ as equation (\ref{eq:F}). Since $\hat{\omega}=-f^0_z+jf^1_z$ and $f_z=\frac{\varphi_u-\psi_v}{2}+j\frac{\varphi_u+\psi_v}{2}$, we obtain
\begin{align*}
\langle \hat{\omega}\rangle^2&=(-f^0_z+jf^1_z)(-f^0_z-jf^1_z)\\
&=(-\varphi_u^0+\varphi_u^1)(\psi_v^0+\psi_v^1).
\end{align*}
For arbitrary $\theta\in \mathbb{R}$, let us define
\begin{align*}
\tilde{f}:= \left(
    \begin{array}{ccc}
      1 &0  & 0 \\
     0  & \cos \theta & -\sin \theta \\
      0 & \sin \theta & \cos \theta 
 \end{array}
  \right)f \quad \text{and}\quad \hat{\omega}_{\theta}:=-\tilde{f}^0_z+j\tilde{f}^1_z.
\end{align*}
Next we prove that there exists a $\theta$ such that $\langle \hat{\omega}_{\theta}\rangle^2 \neq 0$ at $p$. A similar computation as above shows that
\begin{equation*}
\langle \hat{\omega}_{\theta}\rangle^2=(-\varphi_u^0+\varphi_u^1\cos \theta-\varphi_u^2\sin \theta)(\psi_v^0+\psi_v^1\cos \theta-\psi_v^2\sin \theta).
\end{equation*}
Note that the assumption $\langle \hat{\omega}(p)\rangle^2=0$ is equivalent to the condition
\begin{equation*}
-\varphi_u^0(p)+\varphi_u^1(p)=0\quad \text{or}\quad \psi_v^0(p)+\psi_v^1(p)=0.
\end{equation*}
In the former case, $\varphi_u^2(p)=0$ because $F$ is a Lorentzian null map. Since $f$ is a minface, we get $\varphi_u^0(p)=\varphi_u^1(p)\neq0$ and

\begin{equation}\label{eq:omega-theta}
\langle \hat{\omega}_{\theta}(p)\rangle^2=\varphi_u^0(p)(-1+\cos \theta)(\psi_v^0(p)+\psi_v^1(p)\cos \theta-\psi_v^2(p)\sin \theta).
\end{equation}
Let us consider the quantity $\langle \hat{\omega}_{\pi}(p)\rangle^2=-2\varphi_u^0(p)(\psi_v^0(p)-\psi_v^1(p))$. If it is non-zero then the proof is completed. We consider the case $\langle \hat{\omega}_{\pi}(p)\rangle^2=0$. Again, we can see that $\psi_v^2(p)=0$ and $\psi_v^0(p)=\psi_v^1(p)\neq 0$. The equation (\ref{eq:omega-theta}) can be written as
\begin{equation*}
\langle \hat{\omega}_{\theta}(p)\rangle^2=\varphi_u^0(p)(-1+\cos \theta)\psi_v^0(p)(1+\cos \theta),
\end{equation*}
and hence we can take a $\theta$ such that $\langle \hat{\omega}_{\theta}(p)\rangle^2\neq0$. The proof for the case that $\psi_v^0(p)+\psi_v^1(p)=0$ is similar. Therefore we can take $\omega$ and $g$ as paraholomorphic $1$-form and paraholomorphic function near $p$.

Next let us prove the equations (\ref{eq:p-Weierstrass}) and (\ref{eq:p-Weierstrass_real1}). By a straightforward computation, we obtain 
\begin{equation}\label{eq:f_z}
f_zdz=\frac{1}{2}(-1-g^2,j(1-g^2),2g)\omega,
\end{equation}
and hence we obtain the equation (\ref{eq:p-Weierstrass}). For the null coordinates $u$ and $v$, $f$ can be written as $f_z=\frac{f_u-f_v}{2}+j\frac{f_u+f_v}{2}$. By (\ref{eq:f_z}), we get the relation
\begin{equation*}
f_u-f_v=\Re(-1-g^2,j(1-g^2),2g)\hat{\omega}.
\end{equation*}
By using the decompositions (\ref{decomp1}) and (\ref{decomp2}), the equation above can be written
\begin{equation*}
  f_u-f_v=\frac{\hat{\omega}_1}{2}\left( -1-(g_1)^2,1-(g_1)^2, 2g_1 \right)+\frac{\hat{\omega}_2}{2}\left( -1-(g_2)^2,-1+(g_2)^2, 2g_2 \right),
\end{equation*}
and hence 
\begin{equation*}
  f_u=\frac{\hat{\omega}_1}{2}\left( -1-(g_1)^2,1-(g_1)^2, 2g_1 \right) \text{and } f_v=\frac{\hat{\omega}_2}{2}\left( 1+(g_2)^2,1-(g_2)^2, -2g_2 \right).
\end{equation*}
By integrating the derivative $df=f_udu+f_vdv$, we obtain the desired equation (\ref{eq:p-Weierstrass_real1}). 
\end{proof}

\begin{remark}
The formula as mentioned above is valid locally, that is, we cannot choose the function $g$ in Fact \ref{theorem:p-Weierstrass} as a paraholomorphic function globally. However, the notion of {\it parameromorphic function} was introduced in \cite{T} and by using it, Takahashi gave the same formula as (\ref{eq:p-Weierstrass}) with a paraholomorphic 1-form $\omega$ and a parameromorphic function $g$ which are defined on the universal cover $\tilde{\Sigma}$ of $\Sigma$. In this paper we only need the formulas (\ref{eq:p-Weierstrass}) and (\ref{eq:p-Weierstrass_real1}) near each point to discuss the local behavior of the Gaussian curvature, and hence we can always take the function $g$ as a paraholomorphic function locally.
\end{remark}

\begin{remark}
It should be remarked that Magid \cite{Magid} originally proved a representation formula using null curves similar to (\ref{eq:p-Weierstrass_real1}) away from singular points.
\end{remark}

In \cite{T}, the pair $(g,\omega)$ and the quadruple $(g_1, g_2,\omega_1, \omega_2)$ were called ({\it paraholomorphic}) {\it Weierstrass data} and {\it real Weierstrass data}, respectively. The imaginary part
\begin{equation}\label{eq:conjugate}
  f^*(z)
  :=\Im \int^z_{z_0}\left( -1-g^2, j(1-g^2), 2g \right)\omega
\end{equation}
also gives a minface which is called the {\it conjugate minface} of $f$. The conjugate minface is defined on $\tilde{\Sigma}$ and corresponding to a minface with the Weierstrass data $(g,j\omega)$ or the real Weierstrass data $(g_1, g_2,\omega_1, -\omega_2)$.

In the end of the paper, we give a proof of Fact \ref{Fact: T-lemma} and dualities of singular points on minfaces, which were given in \cite{T}. To prove Fact \ref{Fact: T-lemma}, we use the following criteria for cuspidal edges, swallowtails and cuspidal cross caps:
\begin{fact}[\cite{KRSUY}]\label{Fact: KRSUY}
Let $f: U \longrightarrow \mathbb{R}^3$ be a front and $p\in U$ a non-degenerate singular point of $f$. Take a singular curve $\gamma=\gamma(t)$ with $\gamma(0)=p$ and a vector field of null directions $\eta(t)$. Then
\begin{itemize}
\item[(i)] $p$ is a cuspidal edge if and only if $\det\left(\gamma'(0), \eta(0)\right)\neq 0$.
\item[(ii)] $p$ is a swallowtail if and only if $\det\left(\gamma'(0), \eta(0)\right)=0$ and $ \left. \frac{d}{dt}\det\left(\gamma'(t), \eta(t)\right)\right|_{t=0}\neq0$.
\end{itemize}
\end{fact}

\begin{fact}[\cite{FSUY}]\label{Fact: FSUY}
Let $f: U \longrightarrow \mathbb{R}^3$ be a frontal and $p\in U$ a non-degenerate singular point of $f$. Take a singular curve $\gamma=\gamma(t)$ with $\gamma(0)=p$ and a vector field of null directions $\eta(t)$. Then $p$ is a cuspidal cross cap if and only if \begin{center}$\det\left(\gamma'(0), \eta(0)\right)\neq0$, $\det\left(df(\gamma'(0)), n(0), dn(\eta(0))\right)=0$ and $\left.\frac{d}{dt}\det\left(df(\gamma'(t)), n(t), dn(\eta(t))\right)\right|_{t=0}\neq0$.\end{center}
\end{fact}

\begin{proof}[Proof of Fact \ref{Fact: T-lemma}] 
Since the singular set on $U$ is written by $\{p\in U\mid g_1(p)g_2(p)=1\}$, the singular curve $\gamma(t)=(\gamma_1(t), \gamma_2(t))$ near the non-degenerate singular point $p=\gamma(0)$ satisfies $g_1(\gamma_1(t))g_2(\gamma_2(t))=1$. Taking the derivative, we obtain
\begin{equation*}
g'_1\gamma'_1g_2+g_1g'_2\gamma'_2=0.
\end{equation*}
By using the equality $g_1g_2=1$ on the singular set, we can parametrize $\gamma$ as
\begin{equation*}
\gamma'(t)=\frac{g'_2}{g_2}\left(\frac{\partial}{\partial u}\right)_{\gamma(t)}-\frac{g'_1}{g_1}\left(\frac{\partial}{\partial v}\right)_{\gamma(t)}
\end{equation*}
and by the first equation (\ref{eq:eta_nu}), we obtain
\begin{equation}\label{eq:gamma_eta}
\det(\gamma'(t), \eta(t))=\frac{g_1'}{g_1^2\hat{\omega}_1}+\frac{g_2'}{g_2^2\hat{\omega}_2}. 
\end{equation}
By (i) of Fact \ref{Fact: KRSUY} and (iii) of Fact \ref{Lemma T-2}, we have proved the claim (i). By (\ref{eq:gamma_eta}), we can compute
\begin{align}\label{eq:dgamma_eta}
\left.\frac{d}{dt}\det\left(\gamma'(t), \eta(t)\right)\right|_{t=0}&=\left(\frac{g'_1}{g^2_1\hat{\omega}_1}\right)'\gamma'_1+\left(\frac{g'_2}{g^2_2\hat{\omega}_2}\right)'\gamma'_2\nonumber \\ 
&=\left(\frac{g'_1}{g^2_1\hat{\omega}_1}\right)'\frac{g'_2}{g_2}-\left(\frac{g'_2}{g^2_2\hat{\omega}_2}\right)'\frac{g'_1}{g_1}. 
\end{align}
By the equations (\ref{eq:gamma_eta}), (\ref{eq:dgamma_eta}), (ii) of Fact \ref{Fact: KRSUY} and (iii) of Fact \ref{Lemma T-2}, we obtain the claim (ii). Next we prove the claim (iii). After a straightforward computation we obtain 
\begin{equation}\label{eq:det_gamma_n_dn}
\det \left(df(\gamma'(t)), n(t),dn(\eta(t))\right)=\alpha \left(\frac{g_1'}{g_1^2\hat{\omega}_1}-\frac{g_2'}{g_2^2\hat{\omega}_2}\right),
\end{equation}
where $\alpha=\alpha(t)=-\frac{\hat{\omega}_1\hat{\omega}_2}{2}\left(\frac{g_1'}{g_1^2\hat{\omega}_1}+\frac{g_2'}{g_2^2\hat{\omega}_2}\right)$. Since cuspidal cross caps are non-degenerate singular points on frontals, we always assume that $\alpha(0)\neq0$.

Moreover,
\begin{align}\label{eq:ddet_gamma_n_dn}
&\left.\frac{d}{dt}\det \left(df(\gamma'(t)), n(t),dn(\eta(t))\right)\right|_{t=0} \nonumber \\
&=\alpha'(0) \left(\frac{g_1'}{g_1^2\hat{\omega}_1}-\frac{g_2'}{g_2^2\hat{\omega}_2}\right)+\alpha(0)\left[\left(\frac{g'_1}{g^2_1\hat{\omega}_1}\right)'\frac{g'_2}{g_2}+\left(\frac{g'_2}{g^2_2\hat{\omega}_2}\right)'\frac{g'_1}{g_1}\right].
\end{align}
By the equations (\ref{eq:gamma_eta}), (\ref{eq:det_gamma_n_dn}), (\ref{eq:ddet_gamma_n_dn}) and Fact \ref{Fact: FSUY}, we have proved the claim (iii).
\end{proof}

As a corollary of Fact \ref{Fact: T-lemma}, we obtain the following dualities of singular points corresponding to results for maxfaces in \cite{FSUY} and \cite{UY}. It is known that these kinds of dualities also hold for other surfaces, see also \cite{H,IS,OT}.
\begin{fact}[\cite{T}]\label{duality}
Let $f: \Sigma \longrightarrow \mathbb{L}^3$ be a minface and $p\in \Sigma$ a singular point.
\begin{itemize}
\item[(i)] A singular point $p$ of a minface $f$ is a cuspidal edge if and only if $p$ is a cuspidal edge of its conjugate minface $f^*$.
\item[(ii)] A singular point $p$ of a minface $f$ is a swallowtail (resp. cuspidal cross cap) if and only if $p$ is a cuspidal cross cap (resp. swallowtail) of its conjugate minface $f^*$.
\end{itemize}
\end{fact}


 \end{document}